\documentclass[10pt]{article}

\usepackage{amssymb,latexsym,amsmath,enumerate,verbatim,amsfonts,amsthm,microtype}
\usepackage{color} 
\usepackage{graphicx}
\usepackage{algorithm,algorithmic}
\usepackage[active]{srcltx}
\usepackage[colorlinks,
            linkcolor=blue,
            anchorcolor=blue,
            citecolor=blue]{hyperref}
\usepackage{booktabs}
\usepackage{cite}
\usepackage{cases}
\usepackage{mathtools}
\usepackage[figuresright]{rotating}
\usepackage{microtype}
\usepackage [latin1]{inputenc}
\usepackage[framemethod=tikz]{mdframed}
\mdfsetup{%
  skipbelow=4pt,
  skipabove=8pt,
  linewidth=1.25pt,
  backgroundcolor=gray!10,
  userdefinedwidth=\textwidth,
  roundcorner=10pt,
}

\allowdisplaybreaks[2]

\newtheorem{lemma}{Lemma}[section]
\newtheorem{definition}{Definition}[section]

\newtheorem{theorem}{Theorem}[section]

\newtheorem{proposition}{Proposition}[section]
\newtheorem{remark}{Remark}[section]

\newtheorem{assumption}{Assumption}[section]

\numberwithin{equation}{section}

\def\R{{\mathbb{R}}}

\def\Argmin{\mathop{\rm Arg\,min}}

\def\IR{{\bf IR}^{\ell_1}_{\ell_2}}

\title{\sf Doubly iteratively reweighted algorithm for constrained compressed sensing models}

\author{
Shuqin Sun
\thanks{
School of Mathematics Education, China West Normal University, Nanchong, Sichuan, People's Republic of China.
Email: \texttt{sunshuqinsusan@163.com}
}
\and
Ting Kei Pong
\thanks{
Department of Applied Mathematics, the Hong Kong Polytechnic University, Hong Kong, People's Republic of China.
		This author was supported in part by Hong Kong Research Grants Council PolyU153000/20p.
		E-mail: \texttt{tk.pong@polyu.edu.hk}
}
}

\date{June 16, 2022}

\begin{document}

\maketitle

\begin{abstract}
  We propose a new algorithmic framework for constrained compressed sensing models that admit nonconvex sparsity-inducing regularizers including the log-penalty function as objectives, and nonconvex loss functions such as the Cauchy loss function and the Tukey biweight loss function in the constraint. Our framework employs iteratively reweighted $\ell_1$ and $\ell_2$ schemes to construct subproblems that can be efficiently solved by well-developed solvers for basis pursuit denoising such as SPGL1 \cite{Michael2008}. We propose a new termination criterion for the subproblem solvers that allows them to return an \emph{infeasible} solution, with a suitably constructed {\em feasible} point satisfying a descent condition. The feasible point construction step is the key for establishing the well-definedness of our proposed algorithm, and we also prove that any accumulation point of this sequence of feasible points is a stationary point of the constrained compressed sensing model, under suitable assumptions. Finally, we compare numerically our algorithm (with subproblems solved by SPGL1 or the alternating direction method of multipliers) against the SCP$_{\rm ls}$ in \cite{YuLP21} on solving constrained compressed sensing models with the log-penalty function as the objective and the Cauchy loss function in the constraint, for badly-scaled measurement matrices. Our computational results show that our approaches return solutions with better recovery errors, and are always faster.
\end{abstract}

\section{Introduction}\label{sec1}

Compressed sensing \cite{CaT05,FouRau13} is the problem of recovering (approximately) sparse signals from (possibly noisy) measurements that have been compressed for fast transmission.
A classical model for compressed sensing is the basis pursuit denoising, which is minimizing the $\ell_1$ norm subject to a constraint concerning noisy measurements:
\begin{equation}\label{noisyBP}
\begin{array}{rl}
  \min\limits_{x\in \R^n}& \|x\|_1\\
  {\rm s.t.} & \|Ax - b\|\le \underline\sigma,
\end{array}
\end{equation}
where $A\in \R^{m\times n}$ is the measurement matrix, $b\in \R^m$ is the noisy measurement and $\underline\sigma \in [0,\|b\|)$ is a parameter that allows one to incorporate prior knowledge (if any) of the noise level. Model \eqref{noisyBP} is a convex optimization problem and many efficient off-the-shelf solvers have been developed over the past decade for solving it. These include the SPGL1 \cite{Michael2008} that suitably applies the spectral projected gradient method for solving a sequence of $\ell_1$ constrained optimization problems, the YALL1 \cite{YangZhang11} that applies the alternating direction method of multipliers (ADMM) to a suitable reformulation of \eqref{noisyBP}, and a specialized routine in the general purpose first-order method solver package TFOCS \cite{BCG11}, which applies Nesterov's smoothing and acceleration techniques \cite{Nes83,Nesterov2004,Nes2006}, to name but a few. These solvers, taking advantage of convex duality theory and proximal mapping computations, can deal with problems of reasonably large dimension efficiently.

While the convex model \eqref{noisyBP} has been widely used for sparse recovery, it has now become a folklore (see, for example, \cite{RC2007}) that nonconvex sparsity inducing regularizers such as the log-penalty function \cite{CanWakBoyd08,Nikolovalogpenal08} can be used in place of the $\ell_1$ norm in the objective of \eqref{noisyBP} to better induce sparsity in the solution. In this regard, the log-penalty function together with an iteratively reweighted $\ell_1$ (IRL$_1$) technique was introduced in \cite{CanWakBoyd08} for {\em noiseless} compressed sensing, and the IRL$_1$ technique resulted in a \emph{sequence} of convex optimization problems that minimize {\em weighted} $\ell_1$ norms over an affine set. Another generalization for \eqref{noisyBP} is to replace the $\ell_2$ norm in the constraint by other loss functions to reflect different noise models or robustness requirements in the noisy measurement $b$; see, for example, \cite{SAK1985,AZ2012,RE2010,robustreg2014,RE2016}. Concrete examples of alternative loss functions include the Cauchy loss function \cite{RE2010}, the Huber loss function \cite{Huber1964}, and the Tukey biweight loss function \cite{Tukey1977}, again to name but a few. We list in Table~\ref{Table1} as $l(y)$ some commonly used loss functions, where $\delta>0$ in the table and the function ${\cal I}_\delta:\R_+\to \R$ is defined by
\[
{\cal I}_\delta(t):=
\begin{cases}
   1& {\rm if}\ 0\leq t\leq{\delta},\\
   0& {\rm if}\ t>{\delta},
   \end{cases}
   \]
and the Geman-McClure, Welsh and Pseudo-Huber loss functions are written as in \cite{Barron2019} for notational consistency.
Notice that all the functions $l$ in Table~\ref{Table1} satisfy $l(y) = \sum_{i=1}^m \phi(y_i^2)$ with the corresponding $\phi$.

\begin{table}[h]
\begin{center}
\caption{Examples of loss functions}\label{Table1}
\begin{tabular}{|l|l|l|}
\hline
\textbf{Loss functions}&\textbf{Expression of $l(y)$}&\textbf{$\phi(t)$}\\
\hline
\small{Cauchy \cite{RE2010,RE2016}}&\small{$\displaystyle\sum_{i=1}^m\log\left(1+\frac{y_i^2}{\delta^{2}}\right)$}&\small{$\displaystyle\log\left(1+\frac{t}{\delta^{2}}\right)$}\\
\hline
\small{Geman-McClure \cite{GM1985}}&\small{$\displaystyle\sum_{i=1}^m \frac{2y_i^2}{y_i^2 + 4\delta^2}$}&\small{$\displaystyle\frac{2t}{t + 4\delta^2}$}\\
\hline
\small{Welsh \cite{welsh1978}}&\small{$\displaystyle\sum_{i=1}^m\left(1-\exp\left(-\frac{y_i^2}{2\delta^2}\right)\right)$}&\small{$\displaystyle1-\exp\left(-\frac{t}{2\delta^2}\right)$}\\
\hline
\small{Pseudo-Huber \cite{pesoHuber1997}}&\small{$\displaystyle\sum_{i=1}^m\left(\sqrt{1+\frac{y_i^2}{{\delta}^{2}}}-1\right)$}&\small{$\displaystyle \sqrt{1+\frac{t}{{\delta}^{2}}}-1$}\\
\hline
\small{Huber \cite{Huber1964}}
&\small{$\displaystyle\sum_{i=1}^m\left(\frac{y_i^2}{2}{\cal I}_\delta(|y_i|)+\delta\left(|y_i|-\frac{\delta}{2}\right)(1-{\cal I}_\delta(|y_i|))\right)$}
&\small{$\begin{cases}
\frac{t}{2}&{\rm if}\ {\sqrt{t}\leq{\delta}},\\
\delta(\sqrt{t}-\frac{\delta}{2})&{\rm if}\ \sqrt{t}>{\delta}\end{cases}$}\\
\hline
\small{Tukey biweight \cite{Tukey1977}}
&\small{$\displaystyle\sum_{i=1}^m\left(\frac{\delta^{2}}{6}\left[1-\left(1-{\frac{y_i^2}{\delta^2}}\right)^{3}{\cal I}_\delta(|y_i|)\right]\right)$}
&\small{$\begin{cases}
 {\frac{\delta^{2}}{6}}(1-(1-{\frac{t}{\delta^2}})^{3})&{\rm if}\ {\sqrt{t}\leq{\delta}},\\\frac{\delta^{2}}{6}&{\rm if}\ \sqrt{t}>{\delta}\end{cases}$}\\
\hline
\end{tabular}
\end{center}
\end{table}

In this paper, we consider the following optimization problem:
\begin{equation}\label{eq:0.1}
  \begin{array}{rl}
    \min\limits_{x\in{{\mathbb{R}}^{n}}} &\displaystyle\sum_{i=1}^{n}\psi(|x_i|)\\
    {\rm s.t.} & \displaystyle{\sum_{i=1}^{m}\phi((b_i-a^T_ix)^2)\leq{\sigma}},\\
  \end{array}
\end{equation}
where the functions $\psi$, $\phi$, the matrix $A\in \R^{m\times n}$ (with the $i$th row being $a_i^T$), the vector $b\in \R^m$, and the parameter $\sigma > 0$ satisfy the following assumption:
\begin{assumption}\label{ass:1.1}
\begin{enumerate}[{\rm (i)}]
  \item The function $\psi:\mathbb{R}_{+}\rightarrow{\mathbb{R}_{+}}$ is continuous and strictly concave with $\psi(0)=0$ and $\lim_{t\rightarrow{\infty}}\psi(t)=\infty$. It is differentiable on $(0,\infty)$ with $\psi'(t)>0$ for all $t>0$, and ${\lim_{t\downarrow{0}}\psi'}(t)$ exists and belongs to $(0,\infty)$.
  \item The function $\phi:\mathbb{R}_{+}\rightarrow{\mathbb{R}_{+}}$ is continuous and concave with $\phi(0)=0$. It is differentiable on $(0,\infty)$ with $\phi'(t)\geq{0}$ for all $t> 0$ and ${\lim_{t\downarrow{0}}\phi'}(t)$ exists and belongs to $(0,\infty)$.
  \item The matrix $A$ has full row rank and $\sigma \in (0,\sum_{i=1}^m\phi(b_i^2))$.
\end{enumerate}
\end{assumption}
\noindent Note that Assumption~\ref{ass:1.1}(iii) implies that $0$ does not belong to the feasible set of \eqref{eq:0.1}. Moreover, Assumptions~\ref{ass:1.1}(ii) and (iii) imply that $A^\dagger b$ belongs to the feasible set of \eqref{eq:0.1}; hence, the feasible set of \eqref{eq:0.1} is nonempty. Finally, this assumption also implies that the right-hand derivative functions $\phi'_+:\R_+\to \R$ and $\psi'_+:\R_+\to \R$ are continuous functions, and that $x\mapsto \sum_{i=1}^{m}\phi((b_i-a^T_ix)^2)$ is continuously differentiable. One can check that Assumption~\ref{ass:1.1} is general enough to include the choice of $\psi(t) = \log(1 + t/\epsilon)$ for some $\epsilon > 0$ (i.e., the log penalty function proposed in \cite{CanWakBoyd08}) and all the $\phi$'s listed in Table~\ref{Table1}.

Observe that, under Assumption~1.1, the constraint function in \eqref{eq:0.1} is smooth. Moreover, letting $\iota_\psi := \lim_{t\downarrow 0}\psi'(t)$, one can observe from Assumption~\ref{ass:1.1} that the function $t \mapsto \iota_\psi |t| - \psi(|t|)$ is convex. Hence, the objective of \eqref{eq:0.1} admits the following difference-of-convex decomposition:
\[
\sum_{i=1}^{n}\psi(|x_i|) = \iota_\psi \sum_{i=1}^{n}|x_i| - \left(\sum_{i=1}^{n}[\iota_\psi |x_i| - \psi(|x_i|)]\right).
\]
As a consequence, one may suitably adapt difference-of-convex based algorithms such as SCP$_{\rm ls}$ \cite{YuLP21} and other variants described in \cite{LeThiTao18} and references therein to solve \eqref{eq:0.1} under Assumption~\ref{ass:1.1}.
However, compared with those various well-established solvers for \eqref{noisyBP} that can take advantage of convex or gauge duality theory as well as Nesterov's smoothing and acceleration techniques for efficiency, techniques for accelerating the aforementioned algorithms for \eqref{eq:0.1} are relatively limited. In view of this, it is tempting to ask the following question:

\begin{center}
\fbox{\em Can we build an algorithmic framework for \eqref{eq:0.1} that leverages solvers for \eqref{noisyBP}?}
\end{center}

Our point of departure here is the IRL$_1$ technique, which was used in \cite{CanWakBoyd08} for solving a variant of \eqref{eq:0.1} in which the constraint in \eqref{eq:0.1} is replaced by an affine one. The basic version of our proposed framework can be described as follows: in each iteration, we replace $\psi$ and $\phi$ in \eqref{eq:0.1} by their affine majorants at the current iterate. This results in a subproblem of the form
\begin{equation}\label{subproblem00}
\begin{array}{rl}
  \min\limits_{x\in \R^n}& \displaystyle\sum_{i=1}^n\psi_+'(|x_i^k|)|x_i|\\
  {\rm s.t.} & \displaystyle\sum_{i=1}^m\{\phi((b_i-a_i^Tx^k)^2)+ \phi_+'((b_i-a_i^Tx^k)^2)[(b_i-a_i^Tx)^2 - (b_i-a_i^Tx^k)^2]\}\le \sigma;
\end{array}
\end{equation}
the next iterate is then generated as a minimizer of \eqref{subproblem00}. This framework can be viewed as a generalization of the iteratively reweighted schemes in \cite{CanWakBoyd08,CharYin08} where we apply the IRL$_1$ technique to the objective of \eqref{eq:0.1} and the iteratively reweighted $\ell_2$ technique in \cite{CharYin08} to the constraint function in \eqref{eq:0.1}.  Since $\phi'_+ \ge 0$ and $\psi'_+ > 0$ thanks to Assumption~\ref{ass:1.1}, as long as $x^k$ is feasible for \eqref{eq:0.1} (so that $\sigma-\sum_{i=1}^m\phi((b_i-a_i^Tx^k)^2)\ge 0$), problems \eqref{subproblem00} can be equivalently transformed into problems of the form \eqref{noisyBP} with suitably defined $A$, $b$ and $\underline \sigma$, and hence can be solved (approximately) by solvers for \eqref{noisyBP}.

While simple and natural, there are issues when it comes to implementing the above framework, and its convergence behavior is also not clear. First, following the above discussion, one has to guarantee feasibility of $x^k$ for \eqref{eq:0.1} so that \eqref{subproblem00} becomes an instance of \eqref{noisyBP}. Although one can show inductively that {\em exact} minimizers of \eqref{subproblem00} will be feasible for \eqref{eq:0.1} as long as $x^0$ is feasible based on the concavity of $\phi$ (see Section~\ref{sec3} for details), guaranteeing feasibility in practice is not a trivial task. This is because \eqref{subproblem00} is only solved inaccurately and approximately in each iteration in practice, and typical solvers such as SPGL1 may return approximate solutions that (slightly) violate the constraint. Second, even if we assume that
\eqref{subproblem00} are solved {\em exactly} for all $k$, the convergence behavior of $\{x^k\}$ is still unclear. Indeed, classical convergence analysis of iteratively reweighted schemes relies heavily on the strong convexity of subproblems or the realization of the algorithm as an instance of (block-)coordinate minimization scheme (see \cite{Lu14} and references therein). It is not clear whether these approaches can be readily adapted to analyze the basic framework above, where the subproblems \eqref{subproblem00} are in general not strongly convex and the constraint sets can vary across iterations.

In this paper, we adjust the aforementioned framework by incorporating an inexact criterion for solving the subproblems \eqref{subproblem00} approximately, and establish the well-definedness and convergence of the sequence generated by the resulting algorithm, under mild assumptions. Note that commonly used inexact criteria in the literature are typically based on $\epsilon$-subdifferential (see, for example, \cite{InexactVilla2013}) or maintaining summable error to the exact proximal mapping (see, for example, \cite{InexactCombettes2005}); adapted to our subproblems \eqref{subproblem00}, these conditions will require {\em exact} projections onto the feasible set. In contrast, our inexact criterion allows the solver of \eqref{subproblem00} to return an infeasible approximate solution $\widetilde x^{k+1}$, but it requires additionally that the objective value at a suitably constructed retraction $x^{k+1}$ of $\widetilde x^{k+1}$ onto the feasible set of \eqref{subproblem00} \emph{does not exceed the objective value at $x^k$ too much}; and we construct the next subproblem based on $x^{k+1}$. We show that every accumulation point of $\{x^k\}$ is a stationary point of \eqref{eq:0.1} under mild assumptions, even though the subproblems \eqref{subproblem00} are not strongly convex in general. We also demonstrate how our proposed inexact criterion can be realized by subproblem solvers such as ADMM and SPGL1. Finally, we compare our proposed framework (equipped with ADMM or SPGL1 as subproblem solvers) and the SCP$_{\rm ls}$ on solving compressed sensing problems modeled as \eqref{eq:0.1} with the log-penalty function as the objective and the Cauchy loss function in the constraint, for badly scaled measurement matrices $A$. In our experiments, our approaches yield solutions with better recovery errors, and are always faster.

The rest of this paper is organized as follows. We present notation and some preliminary materials in Section~\ref{sec2}. In Section~\ref{sec3}, we present our algorithmic framework with the aforementioned inexact criterion, and establish its well-definedness and convergence. We discuss in Section~\ref{sec4} how the inexact criterion in our algorithmic framework can be realized by two popular solvers for \eqref{noisyBP}: an ADMM-based solver and the SPGL1. Numerical results are presented in Section~\ref{sec5}.

\section{Notation and preliminaries}\label{sec2}
Throughout this paper, we let $\mathbb{R}^{n}$ and $\mathbb{R}_{+}^{n}$ denote the Euclidean space of dimension $n$ and its nonnegative orthant, respectively. For an $x\in \R^n$, we let $|x|$ denote the vector whose $i$th entry is $|x_i|$, and let $\|x\|$ denote the norm of $x$; we also let ${\rm Diag}(x)$ denote the $n\times n$ diagonal matrix whose $i$th diagonal entry equals $x_i$. For an $x\in \R^n_+$, we let $\sqrt{x}$ denote the vector whose $i$th entry is $\sqrt{x_i}$.
For two vectors $x,y\in \mathbb{R}^{n}$, we let $x\circ y$ denote their Hadamard (entry-wise) product and write $x\leq{y}$ if $x_{i}\leq{y_i}$ for each $i=1,2,\cdots,n$. Finally, for a symmetric matrix $A$, we let $\lambda_{\min}(A)$ and $\lambda_{\max}(A)$ denote the smallest and largest eigenvalues of $A$, respectively.

We say that an extended-real-valued function $f:\mathbb{R}^{n}\rightarrow(-\infty,+\infty]$ is proper if its domain ${\rm dom}\,f :=\{x\in \R^n:\;f(x)<\infty\}$ is nonempty. A proper function is said to be closed if it is lower semicontinuous. For a proper function $f$, the regular subdifferential and (limiting) subdifferential\cite[Definition 8.3]{variationalanalysis} of $f$ at an $\bar{x}\in {\rm dom}\,f$ are defined, respectively, as
\[
\widehat{\partial}f(\bar{x}):=\left\{v\in{\mathbb{R}^{n}}:\liminf_{x\rightarrow{\bar{x}},x\neq{\bar{x}}}\frac{f(x)-f(\bar{x})-v^T(x-\bar{x})}{\|x-\bar{x}\|}\geq{0}\right\},
\]
and
\[
\partial{f(\bar{x})}:=\left\{v\in{\mathbb{R}^{n}}:\exists{x^{k}}\rightarrow{\bar{x}}~\mbox{and}~v^{k}\in\widehat{\partial}{f(x^{k})}~\mbox{such}~\mbox{that}~f(x^{k})\rightarrow{f(\bar{x})}~\mbox{and}~v^{k}\rightarrow{v}\right\}.
\]
We set $\widehat{\partial}f({x})=\partial{f(x)}=\emptyset$ if $x\notin{\rm{dom}}\, f$ by convention, and write ${\rm{dom}}\, {\partial{f}:=\{x\in \R^n:\;\partial{f(x)}\neq{\emptyset}}\}$.
When $f$ is proper and convex, the limiting subdifferential of $f$ at an $x\in{\rm{dom}}\,f$ reduces to the classical notion of subdifferential in convex analysis, i.e.,
\[
\partial{f}(x)=\left\{\xi\in \R^n:\;\xi^T(y - x)\leq{f(y)-f(x)}\ \ \forall y\in \R^n\right\};
\]
see \cite[Proposition 8.12]{variationalanalysis}.
For a nonempty set $S$, the indicator function $\delta_{S}$ is defined as
\[
\delta_{S}(x):=
\begin{cases}
   0& {\rm if}\ x\in{S},\\
   \infty& {\rm if}\ x\notin{S}.
   \end{cases}
   \]
The normal cone (resp., regular normal cone) of $S$ at an $x\in{S}$ is defined as $N_{S}(x):=\partial{\delta_{S}}(x)$ (resp., $\widehat{N}_{S}(x):=\widehat{\partial}{\delta_{S}}(x)$), and the distance from any $x\in \R^n$ to $S$ is defined as ${\rm dist}(x,S) := \inf_{y\in S}\|x - y\|$.

We now discuss optimality conditions for \eqref{eq:0.1} under Assumption~\ref{ass:1.1}. For notational simplicity, from now on, for \eqref{eq:0.1}, we write
\begin{equation}\label{notation}
\begin{aligned}
  &\Psi(u) := \sum_{i=1}^n\psi(u_i), \ \ \Phi(v) := \sum_{i=1}^m \phi(v_i), \ \ \mathfrak{F}:=\{x\in{\mathbb{R}^n}:\; \Phi((b-Ax)\circ(b-Ax))\leq{\sigma}\},\\
  &\Psi'_+(u) := (\psi'_+(u_1),\ldots,\psi'_+(u_n))\in \R^n, \ \ {\rm and}\ \ \Phi'_+(v) := (\phi'_+(v_1),\ldots,\phi'_+(v_m))\in \R^m;
\end{aligned}
\end{equation}
specifically, $\mathfrak{F}$ denotes the feasible set of \eqref{eq:0.1}. Note that the function $x\mapsto \Phi((b-Ax)\circ(b-Ax))$ is continuously differentiable everywhere under Assumption~\ref{ass:1.1}(ii). We first recall the following standard constraint qualification for $\frak F$, which is a level set of a continuously differentiable function.

\begin{definition}[MFCQ]\label{def:0.3}
Consider \eqref{eq:0.1} under Assumption \ref{ass:1.1} (with notation \eqref{notation}). We say that the Mangasarian-Fromovitz constraint qualifications (MFCQ) holds for \eqref{eq:0.1} if for every $x\in{\mathfrak{F}}$, the following implication holds
\[
\Phi((b-Ax)\circ(b-Ax)) = \sigma \ \ \ \Longrightarrow \ \ \ \sum_{i=1}^{m}\phi'_{+}((b_{i}-a_{i}^{T}x)^2)(b_{i}-a_{i}^{T}x)a_i \neq 0.
\]
\end{definition}

We also consider the following assumption on the choice of $\sigma$ in \eqref{eq:0.1}.
\begin{assumption}\label{ass:0.2}
Consider \eqref{eq:0.1} under Assumption \ref{ass:1.1}. We assume that
\begin{equation*}
\sigma \notin \{k\overline\phi:\; k = 1,\ldots,m\},
\end{equation*}
where $\overline\phi:=\sup_{t\in{\mathbb{R}_{+}}}\phi(t)\in (0,\infty]$.\footnote{Note that $\overline\phi > 0$ because $\lim_{t\downarrow 0}\phi'(t) > 0$ and $\phi(0)= 0$ in view of Assumption~\ref{ass:1.1}(ii).}
\end{assumption}

We show in the next proposition that Assumption~\ref{ass:0.2} is a sufficient condition for MFCQ.
\begin{proposition}[Assumption~\ref{ass:0.2} implies MFCQ]\label{pro:17}
Consider \eqref{eq:0.1} and suppose that Assumptions~\ref{ass:1.1} and \ref{ass:0.2} hold. Then the following statements hold.
\begin{enumerate}[{\rm (i)}]
  \item If $y\in \R^m$ satisfies $\sum_{i=1}^m\phi(y_i^2) = \sigma$, then there exists $i_0$ such that $y_{i_0} \neq 0$ and $\phi'(y_{i_0}^2) > 0$.
  \item The MFCQ holds for \eqref{eq:0.1}.
\end{enumerate}
\end{proposition}
\begin{proof}
{\rm (i)}: Suppose that $\sum_{i=1}^m\phi(y_i^2) = \sigma$. In view of Assumption~\ref{ass:0.2}, we conclude that there must exist $i_0$ such that $0 < \phi(y_{i_0}^2) < \overline\phi$, where $\overline\phi$ was defined in Assumption~\ref{ass:0.2}. In particular, since $\phi(0)=0$, we must have $y_{i_0}\neq 0$. Then $\phi$ is differentiable at $y_{i_0}^2>0$ by Assumption~\ref{ass:1.1}(ii). Since $\phi(y_{i_0}^2) < \overline\phi$ and $\phi$ is concave with nonnegative derivative on $(0,\infty)$, we conclude that $\phi'(y_{i_0}^2) > 0$.

{\rm (ii)}: Fix any $x\in{\mathfrak{F}}$ satisfying $\sum_{i=1}^{m}\phi((b_{i}-a_{i}^{T}x)^{2})=\sigma$. Then we have from (i) that
there exists an $i_0$ such that
\begin{equation}\label{eq:18}
\phi'_{+}({(b_{i_0}-a_{i_0}^{T}x)^2})>0\ \ \ {\rm and}\ \ \ (b_{i_0}-a_{i_0}^{T}x)^2>0.
\end{equation}
Since the matrix $A$ has full row rank, we can find a $\hat{d}\neq{0}$ such that $A\hat{d}=e_{i_0}$, where $e_{i_0}$ is the vector whose $i_0$th entry is one and is zero otherwise.
Let $d:=(b_{i_0}-a_{i_0}^{T}x)\hat{d}$. Then
\[
a_i^Td = (b_{i_0}-a_{i_0}^{T}x)a_i^T\hat{d} = \begin{cases}
  b_{i_0}-a_{i_0}^{T}x & {\rm if}\ i = i_0,\\
  0 & {\rm otherwise}.
\end{cases}
\]
Hence, combining the above display with \eqref{eq:18}, we deduce further that
\begin{align*}
\sum_{i=1}^{m}\phi'_{+}((b_{i}-a_{i}^{T}x)^2)(b_{i}-a_{i}^{T}x)a_{i}^Td=\phi'_{+}((b_{i_0}-a_{i_0}^{T}x)^2)(b_{i_0}-a_{i_0}^{T}x)^{2} > 0;
\end{align*}
in particular, it must hold that $\sum_{i=1}^{m}\phi'_{+}((b_{i}-a_{i}^{T}x)^2)(b_{i}-a_{i}^{T}x)a_{i}\neq 0$. This completes the proof.
\end{proof}

We next recall the following standard notion of stationarity for \eqref{eq:0.1}.
\begin{definition}[Stationary point]\label{def:0.2}
Consider \eqref{eq:0.1} under Assumption \ref{ass:1.1} (with notation \eqref{notation}). An $x\in{\mathbb{R}^{n}}$ is called a stationary point of \eqref{eq:0.1} if there exists $\lambda\in{\mathbb{R}_{+}}$ such that the following conditions are satisfied for $(x,\lambda)$:
\begin{equation}\label{eq:25}
\lambda(\Phi((b-Ax)\circ(b-Ax))-\sigma)=0,
\end{equation}
\begin{equation}\label{eq:26}
\Phi((b-Ax)\circ(b-Ax))\leq{\sigma},
\end{equation}
\begin{equation}\label{eq:27}
0\in\Psi'_{+}(|x|)\circ\partial\|x\|_1-2\lambda{\sum_{i=1}^{m}}\phi'_{+}((b_{i}-a_{i}^{T}x)^{2})(b_{i}-a_{i}^{T}x){a_i}.
\end{equation}
\end{definition}

Finally, we show that, under Assumptions~\ref{ass:1.1} and \ref{ass:0.2}, every local minimizer of \eqref{eq:0.1} is stationary in the sense of Definition~\ref{def:0.2}.
\begin{proposition}
Consider \eqref{eq:0.1} and suppose that Assumptions~\ref{ass:1.1} and \ref{ass:0.2} hold.
If $x^*$ is a local minimizer of \eqref{eq:0.1}, then it is a stationary point of \eqref{eq:0.1}.
\end{proposition}
\begin{proof}
Since $x^*$ is a local minimizer of \eqref{eq:0.1}, it is feasible and thus satisfies \eqref{eq:26} in place of $x$.
Next, using the notation in \eqref{notation} and \cite[Theorem 10.1]{variationalanalysis}, we have
$0\in{\partial(\Psi(|\cdot|)+\delta_{{\frak F}}(\cdot))(x^*)}.$
Noticing that $x\mapsto \Psi(|x|)$ is locally Lipschitz continuous, we deduce further from \cite[Exercise 10.10]{variationalanalysis} that
\begin{align}
0&\in{\partial\Psi(|\cdot|)(x^*)}+\partial{\delta_{{\frak F}}(x^*)}=\Psi'_{+}(|x^*|)\circ\partial\|x^*\|_1+N_{{\frak F}}(x^*),\label{eq:28}
\end{align}
where the equality follows from \cite[Lemma~2.3]{YuPong19} (Applied with $f = 0$ and $C = \R^n$) and the definition of normal cone.
Now, since the MFCQ holds in view of Proposition~\ref{pro:17}, we can obtain from \cite[Theorem 6.14]{variationalanalysis} that
\begin{align*}
&N_{{\frak F}}(x^*)=\left\{-2\lambda{\sum_{i=1}^{m}\phi'_{+}((b_{i}-a_{i}^{T}x^*)^2)(b_{i}-a_{i}^{T}x^*){a_{i}}}:\;
\lambda\in{N_{-{\mathbb{R}_{+}}}(\Phi((b-Ax^*)\circ(b-Ax^*))-\sigma)}\right\}\\
&=\left\{-2\lambda{\sum_{i=1}^{m}\phi'_{+}((b_{i}-a_{i}^{T}x^*)^2)(b_{i}-a_{i}^{T}x^*){a_{i}}}:\;\lambda\geq{0},
\lambda(\Phi((b-Ax^*)\circ(b-Ax^*))-\sigma)=0\right\},
\end{align*}
where the second equality follows from the definition of normal cone.
The desired conclusion now follows upon combining \eqref{eq:28} with the above display.
\end{proof}

\section{Doubly iteratively reweighted algorithm}\label{sec3}

In this section, we present our algorithmic framework for solving \eqref{eq:0.1}, which involves {\em convex} subproblems that minimize (weighted) $\ell_1$ norms subject to (weighted) least squares constraints: as mentioned in the introduction, this kind of convex optimization problems have been widely studied in the literature and there are many well-developed solvers we can take advantage of.

Our approach for solving \eqref{eq:0.1} is motivated by the huge literature of iteratively reweighted techniques for handling functions $\psi$ and $\phi$ that satisfy Assumption~\ref{ass:1.1}; see, for example, \cite{CharYin08,CanWakBoyd08}. The basic idea is to make use of the following majorization inequalities, which are direct consequences of the concavity assumptions on $\psi$ and $\phi$: for all $s$, $t\in \R_+$, it holds that
\begin{equation}\label{major}
  \psi(s) \le \psi(t) + \psi'_+(t)(s-t)\ \ {\rm and}\ \  \phi(s) \le \phi(t) + \phi'_+(t)(s-t).
\end{equation}
Using the simplifying notation in \eqref{notation}, we now outline our approach for solving \eqref{eq:0.1} under Assumptions~\ref{ass:1.1} and \ref{ass:0.2}.
Suppose we start with an $x^k\in {\frak F}$ in the $k$th iteration. We then construct the following subproblem:
\begin{equation}\label{subproblem}
\begin{array}{rl}
  \min\limits_{x\in \R^n}& \displaystyle\sum_{i=1}^n\psi_+'(|x_i^k|)|x_i|\\
  {\rm s.t.} & \displaystyle\Phi((b-Ax^k)\circ(b-Ax^k)) + \sum_{i=1}^m\phi_+'((b_i-a_i^Tx^k)^2)[(b_i-a_i^Tx)^2 - (b_i-a_i^Tx^k)^2]\le \sigma.
\end{array}
\end{equation}
Notice that the feasible set of \eqref{subproblem} is nonempty (and contains $x^k$) because $x^k\in {\frak F}$. Moreover, since $\psi_+'(t) > 0$ for all $t\in \R_+$ by assumption, we see that the solution set of \eqref{subproblem} is nonempty and we then define $x^{k+1}$ to be any optimal solution of \eqref{subproblem}. Since $x^{k+1}$ is in particular feasible for \eqref{subproblem}, we deduce from \eqref{major} that
\[
\begin{aligned}
&\Phi((b-Ax^{k+1})\circ(b-Ax^{k+1})) \\
&\le \Phi((b-Ax^k)\circ(b-Ax^k)) + \sum_{i=1}^m\phi_+'((b_i-a_i^Tx^k)^2)[(b_i-a_i^Tx^{k+1})^2 - (b_i-a_i^Tx^k)^2]\le \sigma,
\end{aligned}
\]
showing that $x^{k+1}\in {\frak F}$. Inductively, we can obtain a sequence $\{x^k\}$ feasible for \eqref{eq:0.1} with each $x^{k+1}$ being a solution of the convex problem \eqref{subproblem}: as we mentioned before, since problem \eqref{subproblem} is minimizing a weighted $\ell_1$ norm subject to a weighted least squares constraint, there are many readily available solvers for approximately solving it.

While the scheme described above looks simple and natural, it is far from a practical algorithm, and there are issues in terms of both implementation details and theoretical analysis.
\begin{itemize}
  \item The discussion above assumes that $x^{k+1}$ is a global minimizer of \eqref{subproblem} while in practice \eqref{subproblem} can only be {\em solved approximately} by iterative solvers. Note that typical algorithms for solving \eqref{subproblem} approximately (such as the alternating direction method of multipliers and the SPGL1 that we will describe in Section~\ref{sec4}) usually return an {\em infeasible} approximate solution. However, the feasibility of $x^{k+1}$, as noted above, is used in guaranteeing the nonemptiness of the feasible set of the $(k+1)$th subproblem, and hence the well-definedness of the subproblem. It is not immediately clear how the next subproblem can be constructed if $x^{k+1}$ is infeasible.
  \item From the theoretical perspective, even if we assume \eqref{subproblem} are solved {\em exactly} for all $k$, the convergence behavior of $\{x^k\}$ is still unclear. Indeed, classical convergence analysis of iteratively reweighted schemes either makes use of the strong convexity of the subproblem or the identification of the algorithm as a (block-)coordinate minimization scheme applied to a suitable potential function; see, for example, \cite{Lu14} and references therein. The subproblem \eqref{subproblem} is not strongly convex, and it is not obvious whether the approach outlined above is related to coordinate minimization since the constraint set of subproblem \eqref{subproblem} cannot be written as the product of (one-dimensional) intervals and is changing from iteration to iteration. It is not clear how to extend classical convergence analysis to study the above scheme.
\end{itemize}
In the next subsection, we will present our proposed algorithm, which is based on {\em inexactly} solving \eqref{subproblem} in each iteration, with {\em explicitly} specified termination criteria for the subproblem solvers.  We call our algorithm ``doubly iteratively reweighted algorithm with inexact subproblems" ($\IR$ for short). Convergence of the sequence generated will be established under suitable assumptions.

\subsection{Doubly iteratively reweighted algorithm with inexact subproblems}

Our doubly iteratively reweighted algorithm with inexact subproblems ($\IR$) is presented as Algorithm~\ref{algorithm 1} below. In this algorithm, in each iteration, the subproblem \eqref{subproblem2} (which is the same as \eqref{subproblem}) is solved approximately to obtain $(\tilde x^{k+1},\tilde u^{k+1})$ that satisfies \eqref{eq:0077742}, \eqref{eq:0077842} and \eqref{eq:0077831}. The first two conditions \eqref{eq:0077742} and \eqref{eq:0077842} are approximate Karush-Kuhn-Tucker conditions: notice that $\tilde x^{k+1}$ is not necessarily feasible for \eqref{subproblem2}. The ``splitting of variables" in the stopping criterion \eqref{eq:0077742} is inspired by the recent works \cite{YangToh20,YangToh21}, which used a similar strategy for their inexact Bregman proximal gradient algorithm. The third condition \eqref{eq:0077831} requires that the objective value of \eqref{subproblem2} at the feasible point $P_k(\tilde x^{k+1})$ generated by $\tilde x^{k+1}$ does not exceed the objective value at $x^k$ too much. We then update the next iterate as $P_k(\tilde x^{k+1})$.

\begin{algorithm}[h]
\caption{$\IR$: Doubly iteratively reweighted algorithm with inexact subproblems for \eqref{eq:0.1} under Assumptions~\ref{ass:1.1} and \ref{ass:0.2} (See \eqref{notation} for notation)}
\begin{algorithmic}
    \STATE
    {\bf Step 0.} Pick a positive sequence $\{\tau_{k}\}$ with $\tau_k \downarrow 0$ and a summable positive sequence $\{\mu_k\}$. Choose $x^0\in {\frak F}$. Set $k=0$.

    {\bf Step 1.} Compute the following quantities
\begin{align}
&w^{k}=\Psi'_{+}(|x^{k}|),\ \ \ \ \ y^{k}=b-Ax^{k},\ \ \ \ \ v^k = \sqrt{\Phi'_{+}(y^{k}\circ{y^{k}})},\nonumber\\
&A_{k}={\rm Diag}(v^k) A,\ \ \ \ b^{k}=v^k\circ{b},\ \ \ \ \sigma_{k}=\sigma+\|b^k - A_kx^k\|^{2}-\Phi(y^{k}\circ{y^{k}}).\nonumber
\end{align}

    {\bf Step 2.} Pick any $\epsilon_{k}\in (0,\min\{\sigma_{k},\sqrt{\sigma_k},\tau_{k}\}]$ and approximately solve the subproblem
    \begin{equation}\label{subproblem2}
      \begin{array}{rl}
        \displaystyle \min_{x\in \R^n} & \|w^k\circ{x}\|_{1}\\
        {\rm s.t.}& \|{A_{k}x-b^{k}}\|^{2}\leq{\sigma_{k}},
      \end{array}
    \end{equation}
   by finding a pair $(\tilde{x}^{k+1},\tilde{u}^{k+1})$ such that the following three conditions are satisfied:
   \begin{equation}\label{eq:0077742}
   \mathrm{dist}(0,w^{k}\circ{\partial{\|\tilde{x}^{k+1}\|_1}}+A_{k}^{T}N_{{\|\cdot\|^{2}}\leq{\sigma_{k}}}(\tilde{u}^{k+1}))\leq{\epsilon_{k}},
   \end{equation}
   \begin{equation}\label{eq:0077842}
   \|A_{k}\tilde{x}^{k+1}-b^{k}-\tilde{u}^{k+1}\|\leq{\epsilon_{k}},
   \end{equation}
  \begin{equation}\label{eq:0077831}
  \|w^{k}\circ{P_{k}(\tilde{x}^{k+1})}\|_{1}\leq\|{w^{k}\circ{x^{k}}}\|_{1} + \mu_k,
  \end{equation}
  where $P_{k}$ is defined as
  \begin{equation}\label{eq:225164}
     P_{k}(x):=\begin{cases}
           x & {\rm if}\ \|A_kx - b^k\|^2\le \sigma_k,\\
           \left(1-\frac{\sqrt{\sigma_{k}}}{\|A_{k}x-b^{k}\|}\right)A^\dagger b+\frac{\sqrt{\sigma_{k}}}{\|A_{k}x-b^{k}\|}x & {\rm otherwise}.
           \end{cases}
  \end{equation}

  {\bf Step 3.}  Set $x^{k+1}=P_{k}(\tilde{x}^{k+1})$. Update $k\leftarrow{k+1}$ and go to \bf Step 1.

\end{algorithmic}
\label{algorithm 1}
\end{algorithm}

We next argue that $\IR$ is well defined. To this end, it suffices to show that if an $x^k\in {\frak F}$ is given at some iteration $k \ge 0$, then the corresponding subproblem \eqref{subproblem2} has a nonempty feasible set and a pair $(\tilde{x}^{k+1},\tilde{u}^{k+1})$ satisfying \eqref{eq:0077742}, \eqref{eq:0077842} and \eqref{eq:0077831} can be found; moreover, $x^{k+1}\in {\frak F}$. To establish these properties, we first collect some facts concerning $\{\sigma_k\}$ and $P_k$ from $\IR$ in the next lemma.

\begin{lemma}\label{lem:prep}
  Consider \eqref{eq:0.1} under Assumptions \ref{ass:1.1} and \ref{ass:0.2}. Suppose that an $x^k\in {\frak F}$ is generated at the beginning of the $k$th iteration of $\IR$ for some $k\ge 0$. Then the following statements hold:
  \begin{enumerate}[{\rm (i)}]
    \item It holds that $0<\sigma_k\le \sigma$.
    \item For any $x\in \R^n$, it holds that $P_k(x)\in {\frak F}$.
  \end{enumerate}
\end{lemma}
\begin{proof}
{\rm (i)}: Since $x^k \in {\frak F}$, we have $\Phi(y^k\circ y^k) = \Phi((b-Ax^k)\circ (b - Ax^k)) \le \sigma$. If $\Phi(y^k\circ y^k) < \sigma$, then we see immediately from the definition of $\sigma_k$ that $\sigma_k > 0$. Otherwise, suppose that $\Phi(y^k\circ y^k) = \sigma$. Then we see from Proposition~\ref{pro:17}(i) that there exists $i_0$ such that
\begin{equation}\label{eq:1}
(b_{i_0}-a_{i_0}^{T}x^k)^2>0\ \ \ {\rm and}\ \ \
\phi'_{+}((b_{i_0}-a_{i_0}^{T}x^k)^2)>0.
\end{equation}
Using the fact that $\phi'_+(t)\ge 0$ for all $t\in \R_+$, we obtain further that
\begin{align*}
&\Phi(y^{k}\circ{y^{k}}) - \|b^k - A_kx^k\|^{2}=\sum_{i=1}^{m}\phi((b_{i}-a_{i}^{T}x^{k})^{2})-{\sum_{i=1}^{m}\phi'_{+}((b_{i}-a_{i}^{T}x^{k})^{2})(b_{i}-a_{i}^{T}x^{k})^{2}}\\
&\leq{\sum_{i=1}^{m}\phi((b_{i}-a_{i}^{T}x^{k})^{2})-{\phi'_{+}((b_{i_0}-a_{i_0}^{T}x^{k})^{2})(b_{i_0}-a_{i_0}^{T}x^{k})^{2}}}<{\sum_{i=1}^{m}\phi((b_{i}-a_{i}^{T}x^{k})^{2})}=\sigma,
\end{align*}
where the strict inequality follows from \eqref{eq:1}. The above display together with the definition of $\sigma_k$ shows that $\sigma_k > 0$.

Next, we show that $\sigma_{k}\leq{\sigma}$. In fact, one has from \eqref{major} that
\[
\sum_{i=1}^{m}\phi((b_{i}-a_{i}^{T}A^\dagger b)^{2})
\leq\sum_{i=1}^{m}\phi((b_{i}-a_{i}^{T}x^{k})^{2})+\sum_{i=1}^{m}\phi'_{+}((b_{i}-a_{i}^{T}x^{k})^{2})[(b_{i}-a_{i}^{T}A^\dagger b)^{2}-(b_{i}-a_{i}^{T}x^{k})^{2}].
\]
Since $AA^\dagger b = b$ and $\phi(0)=0$, we deduce from the above display that
\begin{align*}
\sum_{i=1}^{m}\phi'_{+}((b_{i}-a_{i}^{T}x^{k})^{2})(b_{i}-a_{i}^{T}x^{k})^{2}-\sum_{i=1}^{m}\phi((b_{i}-a_{i}^{T}x^{k})^{2})\le 0.
\end{align*}
Thus,
\[
\sigma_{k}=\sigma+\sum_{i=1}^{m}\phi'_{+}((b_{i}-a_{i}^{T}x^{k})^{2})(b_{i}-a_{i}^{T}x^{k})^{2}-\sum_{i=1}^{m}\phi((b_{i}-a_{i}^{T}x^{k})^{2})\leq{\sigma}.
\]

{\rm (ii)}: Notice that $A_kA^\dagger b = {\rm Diag}(v^k)AA^\dagger b= {\rm Diag}(v^k) b= b^k$. Using this and the definition of $P_k(x)$, we obtain that
\begin{equation}\label{rel1}
\sum_{i=1}^{m}\phi'_{+}((b_{i}-a_{i}^{T}x^{k})^{2})(b_{i}-a_{i}^{T}P_k(x))^{2} = \|A_k P_k(x) - b^k\|^2 \le \sigma_k
\end{equation}
Now, we have from \eqref{major} that
\begin{equation*}
\begin{aligned}
&\sum_{i=1}^{m}\phi((b_{i}-a_{i}^{T}P_k(x))^{2})\\
& \leq\sum_{i=1}^{m}\phi((b_{i}-a_{i}^{T}x^{k})^{2})+\sum_{i=1}^{m}\phi'_{+}((b_{i}-a_{i}^{T}x^{k})^{2})[(b_{i}-a_{i}^{T}P_k(x))^{2}-(b_{i}-a_{i}^{T}x^{k})^{2}]\\
& \overset{\rm (a)}= \sigma - \sigma_k + \sum_{i=1}^{m}\phi'_{+}((b_{i}-a_{i}^{T}x^{k})^{2})(b_{i}-a_{i}^{T}P_k(x))^{2}\le \sigma,
\end{aligned}
\end{equation*}
where (a) follows from the definition of $\sigma_k$, and the last inequality follows from \eqref{rel1}. This completes the proof.
\end{proof}

\begin{remark}[Well-definedness of $\IR$]\label{rem3.1}
  We now discuss the well-definedness of $\IR$. Suppose that an $x^k\in {\frak F}$ is given at some iteration $k \ge 0$. Then we have $\sigma_k > 0$  according to Lemma~\ref{lem:prep}(i). This means that the corresponding subproblem \eqref{subproblem2} has a nonempty feasible set: indeed, $A^\dagger b$ is a Slater point of the feasible set because $A_kA^\dagger b = b^k$. In addition, the tolerance $\epsilon_k>0$ is well defined. Under the Slater condition and the positivity of $w^k$ and $\epsilon_k$, as we will discuss later in Section~\ref{sec4}, there are many algorithms one can apply to solve \eqref{subproblem2} approximately to obtain a pair $(\tilde{x}^{k+1},\tilde{u}^{k+1})$ that satisfies \eqref{eq:0077742}, \eqref{eq:0077842} and \eqref{eq:0077831}. Finally, Lemma~\ref{lem:prep}(ii) shows that $P_k(\tilde x^{k+1})\in {\frak F}$ and hence $x^{k+1}\in {\frak F}$. These together with an induction argument establish the well-definedness of $\IR$.
\end{remark}

\begin{remark}[The role of $P_k$]\label{rem3.2}
  On passing, we would like to point out that the proof of Lemma~\ref{lem:prep}(i) does not rely on whether the conditions \eqref{eq:0077742}, \eqref{eq:0077842} and \eqref{eq:0077831} are satisfied. Thus, as long as an $x^k\in {\frak F}$ is available, we will have $\sigma_k > 0$ so that the corresponding subproblem \eqref{subproblem2} is well defined. Moreover, regardless of how (in)accurately this subproblem is solved, applying $P_k$ to the approximate solution obtained will return a point in ${\frak F}$, which guarantees the well-definedness of the subproblem in the next iteration. This observation is important for practical implementation when black-box solvers, whose termination conditions are preset / not easy to adjust, are invoked for solving the subproblems \eqref{subproblem2}.
\end{remark}

\begin{proposition}\label{prop:221121}
Consider \eqref{eq:0.1} under Assumptions \ref{ass:1.1} and \ref{ass:0.2}. Let $\{x^k\}$, $\{\tilde x^k\}$, $\{\epsilon_k\}$, $\{\mu_k\}$ and $\{\sigma_k\}$ be as in $\IR$. Then the following statements hold.
\begin{enumerate}[{\rm (i)}]
  \item The sequences $\{x^{k}\}$ and $\{\tilde{x}^{k}\}$ are bounded, and it holds that for all $k$,
  \[
  \Psi(|x^{k+1}|) - \Psi(|x^k|)\le \mu_k.
  \]
  Moreover, the sequence $\{\Psi(|x^k|)\}$ is convergent.
  \item There exists $M>0$ such that for all $k$,
  \[
  \|x^{k+1}-\tilde{x}^{k+1}\|\le\frac{\epsilon_{k}}{\sqrt{\sigma_{k}}}\|A^\dagger b-\tilde{x}^{k+1}\|\leq{\sqrt{\epsilon_{k}}\|A^\dagger b-\tilde{x}^{k+1}\|} \le M \sqrt{\epsilon_k}.
  \]
  \item It holds that $\lim_{k\to\infty} \| |\tilde{x}^{k+1}|-|x^{k}|\|=0.$
\end{enumerate}
\end{proposition}

\begin{proof}

\rm{(i)}: Note that
\[
\begin{aligned}
  \Psi(|x^{k+1}|) - \Psi(|x^k|) & = \sum_{i=1}^{n}(\psi(|x^{k+1}_{i}|)-\psi(|x^{k}_{i}|))\overset{\rm (a)}\leq\sum_{i=1}^{n}\psi'_{+}(|x^{k}_{i}|){(|x^{k+1}_{i}|-|x^{k}_{i}|)}\overset{\rm (b)}\leq \mu_k,
\end{aligned}
\]
where (a) follows from \eqref{major} and (b) holds because of \eqref{eq:0077831}. In particular, for every $k$, we have $\Psi(|x^{k}|)\leq \Psi(|x^{0}|) + \sum_{i=0}^{k-1}\mu_i\le \Psi(|x^{0}|) + \sum_{i=0}^\infty\mu_i <\infty$ since $\{\mu_k\}$ is summable. This together with the assumption $\lim_{t\to\infty}\psi(t)=\infty$ in Assumption \ref{ass:1.1}(i) implies that the sequence $\{x^{k}\}$ is bounded. In addition, the above display together with the summability of $\{\mu_k\}$ implies that
\[
\Psi(|x^{k+1}|) + \sum_{i=k+1}^\infty\mu_i \le \Psi(|x^k|) + \sum_{i=k}^\infty\mu_i,
\]
showing that the sequence $\{\Psi(|x^k|) + \sum_{i=k}^\infty\mu_i\}$ is nonincreasing. Since this sequence is also bounded from below (by zero), it is convergent. This together with the summability of $\{\mu_k\}$ further implies that $\{\Psi(|x^k|)\}$ is convergent.

We next prove the boundedness of $\{\tilde x^{k}\}$. If $\|A_k\tilde{x}^{k+1}-b^k\|^2 \le \sigma_k$ for all large $k$, then $\tilde{x}^{k+1}=x^{k+1}$ for all large $k$ and the boundedness of $\{\tilde{x}^{k}\}$ follows from the boundedness of $\{x^{k}\}$.

Now, suppose that $\|A_k\tilde{x}^{k+1}-b^k\|^2 > \sigma_k$ infinitely often. Define
\[
\mathcal{I}:=\{k:\; \|A_k\tilde{x}^{k+1}-b^k\|^2 > \sigma_k\}\ \ \ {\rm and}\ \ \ \theta_k:= 1-\frac{\sqrt{\sigma_{k}}}{\|A_{k}\tilde{x}^{k+1}-b^{k}\|} \ \ \forall k\in {\cal I}.
\]
Observe that $\{\theta_k:\; k\in{\cal I}\}\subseteq(0,1)$ thanks to the fact that $\sigma_k>0$ (see Lemma~\ref{lem:prep}(i)). We claim that $\sup_{k\in{\mathcal{I}}}{\theta_{k}}<1$. Suppose to the contrary that $\sup_{k\in{\mathcal{I}}}{\theta_{k}}=1$. Then there exists a subsequence $\{{k_{j}}\}\subseteq{\mathcal{I}}$ such that
$\lim_{j\rightarrow\infty}{\theta_{k_{j}}}=1$. In view of the definition of $\theta_k$, this means
\begin{equation}\label{221111}
\lim_{j\rightarrow\infty}\frac{\sqrt{\sigma_{k_{j}}}}{\|A_{k_{j}}\tilde{x}^{k_{j}+1}-b^{k_{j}}\|}={0}.
\end{equation}
On the other hand, using \eqref{eq:0077842} and the fact that $\|\tilde u^{k+1}\|^2\le \sigma_k$ (see \eqref{eq:0077742}), we have $\|A_{k_{j}}\tilde{x}^{k_{j}+1}-b^{k_{j}}\|
\le \|A_{k_{j}}\tilde{x}^{k_{j}+1}-b^{k_{j}} - \tilde u^{k_j+1}\| + \|\tilde u^{k_j+1}\|\leq{\epsilon_{k_{j}}+\sqrt{\sigma_{k_{j}}}}$. Then
\begin{equation*}
1\overset{\rm (a)}=\lim_{j\rightarrow\infty}{\frac{1}{1+\epsilon_{k_j}/\sqrt{\sigma_{k_j}}}}
=\lim_{j\rightarrow\infty}\frac{\sqrt{\sigma_{k_{j}}}}{\sqrt{\sigma_{k_{j}}}+\epsilon_{k_{j}}}
\leq\lim_{j\rightarrow\infty}\frac{\sqrt{\sigma_{k_{j}}}}{\|A_{k_{j}}\tilde{x}^{k_{j}+1}-b^{k_{j}}\|},
\end{equation*}
where (a) holds because $\epsilon_k \downarrow 0$ and $0<\frac{\epsilon_{k}}{\sqrt{\sigma_{k}}}\leq{\frac{\epsilon_{k}}{\sqrt{\epsilon_{k}}}}=\sqrt{\epsilon_{k}}$ (thanks to the fact that $\epsilon_{k}\leq \sigma_{k}$).
This contradicts \eqref{221111}. Thus, we must have $\sup_{k\in{\mathcal{I}}}{\theta_{k}}<1$, which means $r:=\inf_{k\in{\mathcal{I}}}(1-\theta_{k})>0$. This together with the fact that $\{\theta_k:\; k\in{\cal I}\}\subseteq(0,1)$ yields
\begin{equation}\label{haha}
1\leq\frac{1}{1-\theta_{k}}\leq{\frac{1}{r}} < \infty\ \ \ \forall k\in {\cal I}.
\end{equation}
Now, from the definition of $\theta_k$ and $x^{k+1}$, we have $x^{k+1}=\theta_{k} A^\dagger b +(1-\theta_{k}){\tilde{x}^{k+1}}$ for all $k\in {\cal I}$. Thus, we have for $k\in {\cal I}$ that
\[
\|\tilde{x}^{k+1}\|=\left\|\frac{1}{1-\theta_{k}}(x^{k+1}-\theta_{k}A^\dagger b)\right\|\le \frac1r(\|x^{k+1}\| + \|A^\dagger b\|),
\]
where the inequality follows from \eqref{haha}.
Combining this with the boundedness of $\{x^{k}\}$ and the fact that $\tilde x^{k+1} = x^{k+1}$ when $k\notin {\cal I}$ gives the boundedness of $\{\tilde{x}^{k}\}$.

\rm{(ii)}: If $\|A_k\tilde x^{k+1}-b^k\|^2\le \sigma_k$, then the desired conclusion obviously holds because $x^{k+1}=\tilde{x}^{k+1}$.

Now, suppose $\|A_k\tilde x^{k+1}-b^k\|^2> \sigma_k$. Then we have from the definition of $x^{k+1}$ that
\[
x^{k+1}=A^\dagger b+\sqrt{\sigma_{k}}\frac{\tilde{x}^{k+1}-A^\dagger b}{\|A_{k}\tilde{x}^{k+1}-b^{k}\|}.
\]
Therefore,
\begin{equation}\label{haha2}
\begin{aligned}
&\|x^{k+1}-\tilde{x}^{k+1}\|=\left\|A^\dagger b-\tilde{x}^{k+1}+\sqrt{\sigma_{k}}\frac{\tilde{x}^{k+1}-A^\dagger b}{\|A_{k}\tilde{x}^{k+1}-b^{k}\|}\right\|\\
&\overset{\rm (a)}=\left(1-\frac{\sqrt{\sigma_{k}}}{\|A_{k}\tilde{x}^{k+1}-b^{k}\|}\right)\|A^\dagger b-\tilde{x}^{k+1}\|=\frac{\|A_{k}\tilde{x}^{k+1}-b^{k}\|-\sqrt{\sigma_{k}}}{\|A_{k}\tilde{x}^{k+1}-b^{k}\|}\|A^\dagger b-\tilde{x}^{k+1}\|\\
&\overset{\rm (b)}\le\frac{\|A_{k}\tilde{x}^{k+1}-b^{k}\|-\sqrt{\sigma_{k}}}{\sqrt{\sigma_{k}}}\|A^\dagger b-\tilde{x}^{k+1}\|.
\end{aligned}
\end{equation}
where (a) and (b) hold because $\|A_k\tilde x^{k+1}-b^k\|> \sqrt{\sigma_k}$.
Finally, notice that
\begin{align*}
\|A_{k}\tilde{x}^{k+1}-b^{k}\|-\sqrt{\sigma_{k}}&\leq{\|A_{k}\tilde{x}^{k+1}-b^{k}\|-\|\tilde{u}^{k+1}\|}\\
&\leq{\|A_{k}\tilde{x}^{k+1}-b^{k}-\tilde{u}^{k+1}\|}\leq{\epsilon_{k}},
\end{align*}
where the first inequality follows from $\|\tilde{u}^{k+1}\|\leq{\sqrt{\sigma_{k}}}$ (thanks to \eqref{eq:0077742}) and the last inequality holds because of \eqref{eq:0077842}.
Combining this with \eqref{haha2} gives
\[
\|x^{k+1}-\tilde{x}^{k+1}\|\le\frac{\epsilon_{k}}{\sqrt{\sigma_{k}}}\|A^\dagger b-\tilde{x}^{k+1}\|\leq{\sqrt{\epsilon_{k}}\|A^\dagger b-\tilde{x}^{k+1}\|} \le M \sqrt{\epsilon_k},
\]
where the second inequality holds because $\epsilon_k \le \sigma_k$, and the last inequality holds with $M := \|A^\dagger b\| + \sup_{k}\|\tilde x^{k+1}\|$, which is finite thanks to item (i).
This proves (ii).

{\rm (iii)}: Let $\{x^{k_{j}}\}$ and $\{\tilde{x}^{k_{j}+1}\}$ be two convergent subsequences of $\{x^{k}\}$ and $\{\tilde{x}^{k+1}\}$ respectively such that
$\lim_{j\rightarrow{\infty}}x^{k_{j}}=x^{*}$ and $\lim_{j\rightarrow{\infty}}\tilde{x}^{k_{j}+1}=\bar{x}$ for some $x^*$ and $\bar x\in \R^n$. We will show that $|\bar{x}| = |x^{*}|$.

Suppose to the contrary that $|\bar{x}|\neq|x^{*}|$, then
\begin{align}\label{eq:221218}
&\Psi(|\bar{x}|)
<\Psi(|x^{*}|)+\sum_{i=1}^{n}\psi'_{+}(|x^{*}_{i}|)(|\bar x_{i}|-|x^{*}_{i}|)=\Psi(|x^{*}|)+\lim_{j\rightarrow\infty}\sum_{i=1}^{n}\psi'_{+}(|x^{k_{j}}_{i}|)(|{\tilde{x}_{i}^{k_{j}+1}}|-|x^{k_{j}}_{i}|)\nonumber\\
&=\Psi(|x^{*}|)+\lim_{j\rightarrow\infty}\sum_{i=1}^{n}\psi'_{+}(|x^{k_{j}}_{i}|)(|{\tilde{x}_{i}^{k_{j}+1}}|-|{{x}_{i}^{k_{j}+1}}|+|{{x}_{i}^{k_{j}+1}}|-|x^{k_{j}}_{i}|)\nonumber\\
&\le\Psi(|x^{*}|)+\limsup_{j\rightarrow\infty}\sum_{i=1}^{n}\psi'_{+}(|x^{k_{j}}_{i}|)(|{\tilde{x}_{i}^{k_{j}+1}}|-|{{x}_{i}^{k_{j}+1}}|)
+\limsup_{j\rightarrow\infty}\sum_{i=1}^{n}\psi'_{+}(|x^{k_{j}}_{i}|)(|{{x}_{i}^{k_{j}+1}}|-|x^{k_{j}}_{i}|)\nonumber\\
&\leq{\Psi(|x^{*}|)},
\end{align}
where the first inequality follows from the strict concavity of $\psi$, and the last inequality holds because $\limsup_{j\rightarrow\infty}\sum_{i=1}^{n}\psi'_{+}(|x^{k_{j}}_{i}|)(|{\tilde{x}_{i}^{k_{j}+1}}|-|{{x}_{i}^{k_{j}+1}}|) = 0$ (a consequence of items (i) and (ii), the continuity of $\psi'_+$, and the fact that $\epsilon_k\downarrow 0$) and
\eqref{eq:0077831} (which gives $\sum_{i=1}^{n}\psi'_{+}(|x^{k_{j}}_{i}|)(|{{x}_{i}^{k_{j}+1}}|-|x^{k_{j}}_{i}|)\leq\mu_{k_j}$).

On the other hand, notice from $\lim_{j\rightarrow{\infty}}\tilde{x}^{k_{j}+1}=\bar{x}$ and item (ii) that $\lim_{j\rightarrow{\infty}}{x}^{k_{j}+1}=\bar{x}$. Then $\lim_{j\rightarrow{\infty}}(\Psi(|\tilde{x}^{k_{j}+1}|)-\Psi(|{x}^{k_{j}+1}|))=0$ by the continuity of $\Psi$. Moreover, we have
\begin{align}
\Psi(|{x^{*}}|)&=\lim_{j\rightarrow{\infty}}\Psi(|x^{k_{j}}|)=\lim_{j\rightarrow{\infty}}(\Psi(|x^{k_{j}}|)+\Psi(|\tilde{x}^{k_{j}+1}|)-\Psi(|{x}^{k_{j}+1}|)-\Psi(|\tilde{x}^{k_{j}+1}|)+\Psi(|{x}^{k_{j}+1}|))\nonumber\\
&=\lim_{j\rightarrow{\infty}}\Psi(|\tilde{x}^{k_{j}+1}|)+\lim_{j\rightarrow{\infty}}(\Psi(|x^{k_{j}}|)-\Psi(|{x}^{k_{j}+1}|))
-\lim_{j\rightarrow{\infty}}(\Psi(|\tilde{x}^{k_{j}+1}|)-\Psi(|{x}^{k_{j}+1}|))\nonumber\\
&=\Psi(|\bar{x}|),\nonumber
\end{align}
where we also used in the last equality the fact that $\{\Psi(|x^{k}|)\}$ is convergent (see item (i)). This contradicts \eqref{eq:221218}.
Thus, we must have $|\bar{x}|=|x^{*}|$. This completes the proof.
\end{proof}

\begin{remark}[Boundedness of some auxiliary sequences]\label{rem:221211}
Consider \eqref{eq:0.1} under Assumptions \ref{ass:1.1} and \ref{ass:0.2}.
As an immediate consequence of Proposition~\ref{prop:221121}(i), the sequences $\{w^{k}\}$, $\{A_{k}\}$, $\{b^{k}\}$ and $\{\sigma_k\}$ generated by $\IR$ are all bounded. In addition, the sequence $\{A_{k}\tilde{x}^{k+1}-b^{k}\}$ is also bounded.
\end{remark}

From \eqref{eq:0077742}, we see that for every $k$, there exists a $\xi^{k}\in{\mathbb{R}^{n}}$ with $\|\xi^{k}\|\leq{\epsilon_{k}}$ such that
\[
\xi^{k}\in w^{k}\circ{\partial{\|\tilde{x}^{k+1}\|_1}}+A_{k}^{T}N_{{\|\cdot\|^{2}}\leq{\sigma_{k}}}(\tilde{u}^{k+1}).
\]
From the definition of normal cone and noting that $\sigma_k > 0$ (see Lemma~\ref{lem:prep}), we deduce that there exists a ${\tilde{\lambda}}_{k}\geq{0}$ such that
\begin{equation}\label{eq:22180}
\xi^{k}\in{w^{k}\circ{\partial{\|\tilde{x}^{k+1}\|_1}}+{\tilde{\lambda}}_{k}A_{k}^{T}\tilde{u}^{k+1}}\ \ {\rm and}\ \ \tilde\lambda_k(\|\tilde u^{k+1}\|^2 - \sigma_k) = 0.
\end{equation}
Moreover, if we define
\begin{equation}\label{eq:221190}
\tilde{v}^{k+1}:=A_{k}\tilde{x}^{k+1}-b^{k}-\tilde{u}^{k+1},
\end{equation}
then we see from \eqref{eq:0077842} and the definition of $\epsilon_k$ that
\begin{equation}\label{eq:221214}
\|\tilde{v}^{k+1}\|\leq{\epsilon_{k}}\leq\min\{{\sigma_{k}},\sqrt{\sigma_k}\}.
\end{equation}
Furthermore, we can deduce from \eqref{eq:22180} and the definition of $\tilde{v}^{k+1}$ that
\begin{equation}\label{relation00}
\tilde{x}^{k+1}\in \Argmin\left\{{(w^{k})^{T}}|x|-(\xi^{k})^{T}x+0.5\tilde{\lambda}_{k}(\|A_{k}x-b^{k}-\tilde{v}^{k+1}\|^{2}-\sigma_{k})\right\}
\end{equation}
and
\begin{equation}\label{relation01}
\tilde{\lambda}_{k}(\|A_{k}\tilde{x}^{k+1}-b^{k}-\tilde{v}^{k+1}\|^{2}-\sigma_{k})=0.
\end{equation}
The next theorem exploits the above observations to study properties of $\{\tilde \lambda_k\}$ and establish the subsequential convergence of the sequence $\{x^k\}$ generated by $\IR$.

\begin{theorem}[Subsequential convergence of $\IR$]\label{the3.1}
Consider \eqref{eq:0.1} under Assumptions \ref{ass:1.1} and \ref{ass:0.2}. Let $\{x^{k}\}$ and $\{\tilde x^k\}$ be the sequences generated by $\IR$ and let $\tilde\lambda_k$ be defined in \eqref{eq:22180}. Then the following statements hold.
\begin{enumerate}[{\rm (i)}]
  \item It holds that $\liminf_{k\rightarrow{\infty}}{\tilde{\lambda}}_{k}>0$.
  \item We have $\lim_{k\to \infty} \phi'_+((b_i - a_i^Tx^{k})^2)(a_i^T(\tilde x^{k+1} - x^k))=0$ for all $i$.
  \item Every accumulation point $x^*$ of $\{x^{k}\}$ satisfies $\Phi((b-Ax^*)\circ (b-Ax^*))=\sigma$.
  \item The sequence $\{\tilde{\lambda}_{k}\}$ is bounded.
  \item Every accumulation point of $\{x^{k}\}$ is a stationary point of \eqref{eq:0.1}.
\end{enumerate}
\end{theorem}

\begin{proof}

\rm{(i)}: Suppose to the contrary that $\liminf_{k\rightarrow{\infty}}{\tilde{\lambda}}_{k}=0$. Then there exists a subsequence $\{\tilde\lambda_{k_{j}}\}$ of $\{\tilde\lambda_k\}$ such that $\tilde\lambda_{k_{j}}\rightarrow{0}$ with $\tilde\lambda_{k_{j}}\geq0$ for all $j$. Since $\{x^k\}$ and $\{\tilde x^{k}\}$ are bounded in view of Proposition~\ref{prop:221121}(i), by passing to further subsequences if necessary, we assume without loss of generality that $x^{k_j}\rightarrow{x^{*}}$ and ${\tilde {x}^{k_{j}+1}}\rightarrow{\hat{x}}$ for some $x^*$ and $\hat x$.
Passing to the limit as $j\rightarrow{\infty}$ in \eqref{eq:22180} with $k_j$ in place of $k$, and recalling $\epsilon_k\downarrow 0$ and the boundedness of $\{A^T_{k}\tilde u^{k+1}\}$ (see Remark~\ref{rem:221211} and note that $\|\tilde u^{k+1}\|^2\le \sigma_k$), we obtain
\[
0\in\Psi'_{+}(|x^{*}|)\circ{\partial\|\hat{x}\|_1},
\]
which means $0\in{\partial\|\hat{x}\|_1}$ since $\Psi'_{+}(|x^{*}|)>0$. This implies that $\hat{x}=0$. On the other hand, since ${\tilde {x}^{k_{j}+1}}\rightarrow{\hat{x}}$, we have from Proposition~\ref{prop:221121}(ii) and $\epsilon_k\downarrow 0$ that $\lim_{j\to \infty}x^{k_j+1} = \hat x$. Since $\{x^{k}\}\subseteq {\frak F}$ (thanks to Lemma~\ref{lem:prep}(ii)), $0\notin{\mathfrak{F}}$ and ${\frak F}$ is closed, we conclude that $\hat x \neq 0$, leading to a contradiction. Therefore, we must have $\liminf_{k\rightarrow{\infty}}{\tilde{\lambda}}_{k}>0$.

\rm{(ii)}: With $\tilde v^{k+1}$ defined in \eqref{eq:221190} and noting \eqref{eq:221214}, we have
\begin{equation}\label{delta}
\begin{aligned}
\Delta_k&:=
  [(\tilde{v}^{k+1})^T(A_{k}x^{k}-b^{k})]^{2}-\|A_{k}x^{k}-b^{k}\|^{2}(\|\tilde{v}^{k+1}\|^{2}-\sigma_{k})\\
  &= (\sigma_{k} - \|\tilde{v}^{k+1}\|^{2})\|A_{k}x^{k}-b^{k}\|^{2} + [(\tilde{v}^{k+1})^T(A_{k}x^{k}-b^{k})]^{2}\\
  & \ge [(\tilde{v}^{k+1})^T(A_{k}x^{k}-b^{k})]^{2} \ge 0.
\end{aligned}
\end{equation}
Now, define
\begin{equation}\label{q}
  q_k := \begin{cases}
   \displaystyle \frac{(\tilde{v}^{k+1})^T(A_{k}x^{k}-b^{k}) + \sqrt{\Delta_k}}{\|A_kx^k - b^k\|^2} & {\rm if}\ \|A_kx^k - b^k\| > 0,\\
   1 & {\rm otherwise}.
  \end{cases}
\end{equation}
Then we see from \eqref{delta} that $q_k$ is well defined and $q_k \ge 0$ for all $k$. Next, define
\begin{equation}\label{thatx}
t_{k}:=\min{\{1,q_{k}\}}\ \ {\rm and}\ \ \hat{x}^{k}:=t_{k}x^{k}+(1-t_{k})A^\dagger b.
\end{equation}
Then $t_k \in [0,1]$ and we have
\begin{equation}\label{eq:2211210}
\begin{aligned}
&\|A_{k}{\hat{x}^{k}}-b^{k}-\tilde{v}^{k+1}\|^{2}-{\sigma_{k}}\\
=&\|A_{k}(t_{k}x^{k}+(1-t_{k})A^\dagger b)-(t_{k}b^{k}+(1-t_{k})b^{k})-\tilde{v}^{k+1}\|^{2}-{\sigma_{k}}\\
=&\|t_{k}(A_{k}x^{k}-b^{k})+(1-t_{k})(A_{k}A^\dagger b-b^{k})-\tilde{v}^{k+1}\|^{2}-{\sigma_{k}}\\
=&\|t_{k}(A_{k}x^{k}-b^{k})-\tilde{v}^{k+1}\|^{2}-{\sigma_{k}} \leq{0},
\end{aligned}
\end{equation}
where the third equality follows from the fact $AA^\dagger b = b$ and the definitions of $A_{k}$ and $b^{k}$, while the inequality holds because of the definition of $q_k$: indeed, if $A_kx^k - b^k = 0$, the inequality follows directly from \eqref{eq:221214}, while if $A_kx^k - b^k \neq 0$, the inequality holds true because $q_k$ is defined to be the larger root of the quadratic function $t\mapsto \|t(A_{k}x^{k}-b^{k})-\tilde{v}^{k+1}\|^{2}-{\sigma_{k}}$.
Moreover, we have from \eqref{thatx} and the boundedness of $\{x^k\}$ (see Proposition~\ref{prop:221121}(i)) that $\{\hat{x}^{k}\}$ is also bounded.

Let $J:= \{k:\; \|A_kx^k - b^k\| > 0\}$.
Then, writing $\tilde y^k:= A_kx^k-b^k$ for notational simplicity, we have from \eqref{delta}, \eqref{q} and the fact $\epsilon_k\downarrow 0$ that
\begin{eqnarray*}
&&\liminf_{J\ni k\rightarrow{\infty}}{q_{k}}=\liminf_{J\ni k\rightarrow{\infty}}\frac{(\tilde{v}^{k+1})^T(A_{k}x^{k}-b^{k}) + \sqrt{\Delta_k}}{\|A_kx^k - b^k\|^2}\\
&&=\liminf_{J\ni k\rightarrow{\infty}}\frac{(\tilde{v}^{k+1})^T\tilde y^k + \sqrt{(\sigma_{k} - \|\tilde{v}^{k+1}\|^{2})\|\tilde y^k\|^{2} + [(\tilde{v}^{k+1})^T\tilde y^k]^{2}}}{\|\tilde y^k\|^2}\\
&&\overset{\rm (a)}\ge \liminf_{J\ni k\rightarrow{\infty}}\frac{-\|\tilde{v}^{k+1}\|\|\tilde y^k\| + \sqrt{(\sigma_{k} - \|\tilde{v}^{k+1}\|^{2})\|\tilde y^k\|^{2}} - \sqrt{[(\tilde{v}^{k+1})^T\tilde y^k]^{2}}}{\|\tilde y^k\|^2}\\
&& \overset{\rm (b)}\ge \liminf_{J\ni k\rightarrow{\infty}}\frac{-2\|\tilde{v}^{k+1}\|\|\tilde y^k\| + \sqrt{\sigma_{k} - \|\tilde{v}^{k+1}\|^{2}}\|\tilde y^k\|}{\|\tilde y^k\|^2}\\
&& = \liminf_{J\ni k\rightarrow{\infty}}\frac{-2\|\tilde{v}^{k+1}\| + \sqrt{\sigma_{k} - \|\tilde{v}^{k+1}\|^{2}}}{\|\tilde y^k\|}
\overset{\rm (c)}\ge \liminf_{J\ni k\rightarrow{\infty}}\frac{\sqrt{\sigma_k}-3\|\tilde{v}^{k+1}\|}{\|A_kx^k - b^k\|}\\
&&\overset{\rm (d)}\ge  \liminf_{J\ni k\rightarrow{\infty}}(1-3\sqrt{\epsilon_k})\frac{\sqrt{\sigma_k}}{\|A_kx^k - b^k\|}=\liminf_{J\ni k\rightarrow{\infty}}\frac{\sqrt{\sigma_{k}}}{\|A_{k}x^{k}-b^{k}\|}\geq{1},
\end{eqnarray*}
where: (a) follows from the Cauchy-Schwartz inequality and the elementary relation that $\sqrt{a+b}\ge \sqrt{a} - \sqrt{b}$ for all $a$, $b\ge 0$, (b) follows from the Cauchy-Schwartz inequality, (c) follows from the relation $\sqrt{a-b}\ge \sqrt{a} - \sqrt{b}$ for all $a\ge b\ge 0$, (d) holds because $\epsilon_k\le \sigma_k$ so that \eqref{eq:221214} implies $\|\tilde v^{k+1}\|\le \epsilon_k = \sqrt{\epsilon_k}\sqrt{\epsilon_k}\le \sqrt{\epsilon_k}\sqrt{\sigma_k}$, and the last inequality holds because $\|A_{k}x^{k}-b^{k}\|^{2}\leq{\sigma_{k}}$, which is a consequence of the definition of $\sigma_k$ and the fact $x^{k}\in{\mathfrak{F}}$ (see Lemma~\ref{lem:prep}(ii)).
Consequently, we can deduce from the definition of $t_{k}$ in \eqref{thatx} that
$\lim_{k\rightarrow{\infty}}{t_{k}}=1$.
This together with the boundedness of $\{x^k\}$ (see Proposition~\ref{prop:221121}(i)) yields
\begin{align}\label{eq:22185}
\lim_{k\rightarrow{\infty}}\|\hat{x}^{k}-x^{k}\|=\lim_{k\rightarrow{\infty}}\|{t_{k}}x^{k}+(1-t_{k})A^\dagger b-x^{k}\|=0.
\end{align}

With the help of the auxiliary sequence $\{\hat x^k\}$, we are now ready to prove that
\[
\lim_{k\rightarrow{\infty}}{{\phi'_{+}((b_{i}-a_{i}^{T}x^{k})^{2})}(a_{i}^{T}(\tilde{x}^{k+1}-{x}^{k}))}=0.
\]
To this end, we note from \eqref{relation01} that
\begin{equation}\label{relation}
\begin{aligned}
&{(w^{k})^{T}}|\tilde{x}^{k+1}|-(\xi^{k})^{T}\tilde{x}^{k+1}\\
&={(w^{k})^{T}}|\tilde{x}^{k+1}|-(\xi^{k})^{T}\tilde{x}^{k+1}+0.5\tilde{\lambda}_{k}(\|A_{k}\tilde{x}^{k+1}-b^{k}-\tilde{v}^{k+1}\|^{2}-\sigma_{k})\\
&\overset{\rm (a)}\leq{{(w^{k})^{T}}|\hat{x}^{k}|-(\xi^{k})^{T}\hat{x}^{k}+0.5\tilde{\lambda}_{k}(\|A_{k}\hat{x}^{k}-b^{k}-\tilde{v}^{k+1}\|^{2}-\sigma_{k})}-0.5\tilde{\lambda}_{k}\|A_{k}\tilde{x}^{k+1}-A_{k}\hat{x}^{k}\|^{2}\\
&\overset{\rm (b)}\leq{{(w^{k})^{T}}|\hat{x}^{k}|-(\xi^{k})^{T}\hat{x}^{k}}-0.5\tilde{\lambda}_{k}\|A_{k}\tilde{x}^{k+1}-A_{k}\hat{x}^{k}\|^{2},
\end{aligned}
\end{equation}
where (a) follows from \eqref{relation00} and convexity, and (b) follows from \eqref{eq:2211210} and the fact that $\tilde\lambda_k \ge 0$.
Now, using the concavity of $\psi$ and recalling the definition of $\{w^k\}$ in $\IR$, we have
\begin{equation*}
\begin{aligned}
&\Psi{(|x^{k+1}|)}\leq{\Psi(|x^{k}|)+(w^{k})^{T}(|x^{k+1}|-|x^{k}|)}\\
&={\Psi(|x^{k}|)+(w^{k})^{T}(|x^{k+1}|-|\tilde{x}^{k+1}|)+(w^{k})^{T}(|\tilde{x}^{k+1}|-|\hat{x}^{k}|)+(w^{k})^{T}(|\hat{x}^{k}|-|x^{k}|)}\\
&\leq\Psi(|x^{k}|)+(w^{k})^{T}(|x^{k+1}|-|\tilde{x}^{k+1}|)+(\xi^{k})^{T}(\tilde{x}^{k+1}-\hat{x}^{k})-0.5\tilde{\lambda}_{k}\|A_{k}\tilde{x}^{k+1}-A_{k}\hat{x}^{k}\|^{2}\\
& \ \ \ +(w^{k})^{T}(|\hat{x}^{k}|-|x^{k}|),
\end{aligned}
\end{equation*}
where the last inequality follows from \eqref{relation}.
Upon rearranging terms in the above inequality, we obtain further that
\begin{equation}\label{eq:22196}
\begin{aligned}
0.5\tilde{\lambda}_{k}\|A_{k}\tilde{x}^{k+1}-A_{k}\hat{x}^{k}\|^{2}
&\leq\Psi(|x^{k}|)-\Psi{(|x^{k+1}|)}+(w^{k})^{T}(|x^{k+1}|-|\tilde{x}^{k+1}|)\\
&\ \ \ +(\xi^{k})^{T}(\tilde{x}^{k+1}-\hat{x}^{k})+(w^{k})^{T}(|\hat{x}^{k}|-|x^{k}|).
\end{aligned}
\end{equation}
Recall from Proposition~\ref{prop:221121}(i) that $\{\Psi(|x^{k}|)\}$ is convergent. In addition, we have from the boundedness of $\{\hat x^{k}\}$ and $\{\tilde{x}^{k+1}\}$ (see \eqref{thatx} and Proposition~\ref{prop:221121}(i)) and the fact $\lim_{k\to\infty}\xi^k = 0$ (see \eqref{eq:22180} and note that $\epsilon_k\downarrow 0$) that
$\lim_{k\rightarrow{\infty}}(\xi^{k})^{T}(\tilde{x}^{k+1}-\hat{x}^{k})=0$. Thus, passing to the limit as $k\to \infty$ on both sides of \eqref{eq:22196} and noting the boundedness of $\{w^k\}$ (see Remark~\ref{rem:221211}), Proposition~\ref{prop:221121}(ii), item (i) and \eqref{eq:22185}, we conclude that
\[
\lim_{k\rightarrow{\infty}}\|A_{k}\tilde{x}^{k+1}-A_{k}\hat{x}^{k}\|^2=0.
\]
Invoking \eqref{eq:22185} again and noting that $\{A_k\}$ is bounded (see Remark~\ref{rem:221211}), we deduce further that
\begin{equation*}
\lim_{k\rightarrow{\infty}}\|A_{k}\tilde{x}^{k+1}-A_{k}{x}^{k}\|^2=0.
\end{equation*}
This together with the definition of $A_k$ further implies
\[
\lim_{k\rightarrow{\infty}}{\sqrt{{\phi'_{+}((b_{i}-a_{i}^{T}x^{k})^{2})}}(a_{i}^{T}(\tilde{x}^{k+1}-{x}^{k}))}=0 \ \ \ \forall i.
\]
Finally, this relation together with the boundedness of $\{x^k\}$ and the continuity of $\phi'_+$ (and hence the boundedness of $\{\phi'_+((b_i - a_i^Tx^{k})^2)\}$) implies that
\[
\lim_{k\to \infty} \phi'_+((b_i - a_i^Tx^{k})^2)(a_i^T(\tilde x^{k+1} - x^k)) = 0  \ \ \ \forall i.
\]

\rm{(iii)}: Let $x^{*}$ be an arbitrary accumulation point of $\{x^{k}\}$ and let $\{x^{k_j}\}$ be a convergent subsequence such that $\lim_{j\rightarrow{\infty}}x^{k_{j}}=x^{*}$. Since $\{\tilde x^k\}$ is bounded in view of Proposition~\ref{prop:221121}(i), by passing to a further subsequence if necessary, we assume without loss of generality that $\lim_{j\rightarrow{\infty}}\tilde{x}^{k_{j}+1}=\bar{x}$ for some $\bar x$.

Then, for each $i$ satisfying
$\lim_{j\rightarrow{\infty}}{{\phi'_{+}((b_{i}-a_{i}^{T}x^{k_{j}})^{2})}}={\phi'_{+}((b_{i}-a_{i}^{T}x^{*})^{2})}>0$, we have from item (ii) that
\begin{equation}\label{eq:221193}
a_{i}^{T}(\bar{x}-{x}^{*})=\lim_{j\rightarrow{\infty}}{a_{i}^{T}(\tilde{x}^{k_{j}+1}-{x}^{k_{j}})}=0.
\end{equation}
Next, noting item (i), we assume without loss of generality that $\tilde{\lambda}_{k_{j}}>0$ for all sufficiently large $j$. Using this and \eqref{relation01}, we deduce further that for all large $j$,
\begin{align*}
0& =\|A_{k_{j}}\tilde{x}^{k_{j}+1}-b^{k_{j}}-\tilde{v}^{k_{j}+1}\|^{2}-\sigma_{k_{j}}\\
&=\|A_{k_{j}}\tilde{x}^{k_{j}+1}-b^{k_{j}}\|^{2}-2(\tilde{v}^{k_{j}+1})^T(A_{k_{j}}\tilde{x}^{k_{j}+1}-b^{k_{j}})+\|\tilde{v}^{k_{j}+1}\|^{2}-\sigma_{k_{j}}.
\end{align*}
Invoking the definition of $\sigma_k$, $A_k$ and $b^k$, we can rewrite the above relation as
\begin{align*}
0& = \Phi(y^{k_j}\circ y^{k_j})+\sum_{i=1}^{m}\{\phi'_{+}((b_{i}-a_{i}^{T}x^{k_{j}})^{2})[(b_{i}-a_{i}^{T}\tilde{x}^{k_{j}+1})^{2}-(b_{i}-a_{i}^{T}x^{k_{j}})^{2}]\}-\sigma\\
&\ \ \ -2(\tilde{v}^{k_{j}+1})^T(A_{k_{j}}\tilde{x}^{k_{j}+1}-b^{k_{j}})+\|\tilde{v}^{k_{j}+1}\|^{2}.
\end{align*}
Passing to the limit as $j\rightarrow{\infty}$ in the above relation, we have upon noting $\lim_{j\rightarrow{\infty}}{\|\tilde{v}^{k_{j}+1}\|^{2}}=0$ (see \eqref{eq:221214}) and $\lim_{j\rightarrow\infty}(\tilde{v}^{k_{j}+1})^T(A_{k_{j}}\tilde{x}^{k_{j}+1}-b^{k_{j}})=0$ (thanks to Remark~\ref{rem:221211}) that
\begin{equation}\label{eq:221194}
0=\Phi(y^*\circ y^*)+\sum_{i=1}^{m}\{\phi'_{+}((b_{i}-a_{i}^{T}x^{*})^{2})[(b_{i}-a_{i}^{T}\bar{x})^{2}-(b_{i}-a_{i}^{T}x^{*})^{2}]\}-\sigma,
\end{equation}
where $y^* := b - Ax^*$.
Now, if $\phi'_{+}((b_{i}-a_{i}^{T}x^{*})^{2})=0$, then $\phi'_{+}((b_{i}-a_{i}^{T}x^{*})^{2})[(b_{i}-a_{i}^{T}\bar{x})^{2}-(b_{i}-a_{i}^{T}x^{*})^{2}]=0$ holds. Otherwise, if
$\phi'_{+}((b_{i}-a_{i}^{T}x^{*})^{2})>0$, then we still have $\phi'_{+}((b_{i}-a_{i}^{T}x^{*})^{2})[(b_{i}-a_{i}^{T}\bar{x})^{2}-(b_{i}-a_{i}^{T}x^{*})^{2}]=0$ from \eqref{eq:221193}.
Thus, we deduce further from \eqref{eq:221194} that $\Phi(y^*\circ y^*) = \sigma$ as desired.

\rm{(iv)}: Suppose to the contrary that $\{\tilde{\lambda}_{k}\}$ is unbounded. Then there exists a subsequence such that $\tilde\lambda_{k_j}> 0$ for all $j$ and $\lim_{j\rightarrow{\infty}}\tilde{\lambda}_{k_{j}}=\infty$. Since the sequences $\{x^k\}$ and $\{\tilde x^k\}$ are bounded (thanks to Proposition~\ref{prop:221121}(i)), by passing to further subsequences if necessary, we assume without loss of generality that $\lim_{j\rightarrow{\infty}}{x^{k_{j}}}=x^{*}$ and $\lim_{j\rightarrow{\infty}}{\tilde{x}^{k_{j}+1}}=\bar{x}$ for some $x^*$ and $\bar x\in \R^n$.

Dividing both sides of the inclusion in \eqref{eq:22180} by $\{\tilde{\lambda}_{k_j}\}$, we see that
\[
\frac{\xi^{k_j}}{\tilde{\lambda}_{k_j}}\in{\frac{w^{k_j}\circ{\partial{\|\tilde{x}^{k_j+1}\|_1}}}{\tilde{\lambda}_{k_j}}}+A_{k_j}^{T}{\tilde{u}^{k_j+1}}.
\]
Note that $\|\xi^{k}\|\leq{\epsilon_{k}}$, $\{w^k\}$ is bounded in view of Remark \ref{rem:221211}, and the subdifferential of the $\ell_1$ norm is contained in the unit $\ell_\infty$ norm ball. Since $\lim_{j\rightarrow{\infty}}\tilde{\lambda}_{k_{j}}=\infty$, we have upon passing to the limit on both sides of the above display that $\lim_{j\to \infty}A^T_{k_j}\tilde u^{k_j+1} = 0$. We can further rewrite this limit as
\begin{equation}\label{eq:22192}
\lim_{j\rightarrow{\infty}}{\sum_{i=1}^{m}\phi'_{+}((b_{i}-a_{i}^{T}x^{k_{j}})^{2})(b_{i}-a_{i}^{T}\tilde{x}^{k_{j}+1}){a_{i}}}+
A_{k_{j}}^{T}\tilde{v}^{k_{j}+1}=0,
\end{equation}
where $\tilde v^{k+1}$ is defined as in \eqref{eq:221190}.
Now, notice from \eqref{eq:221214} and the boundedness of the sequence $\{A_{k_{j}}^{T}\}$ (thanks to Remark \ref{rem:221211}) that
$\lim_{j\rightarrow{\infty}}A_{k_{j}}^{T}\tilde{v}^{k_{j}+1}=0$.
This together with \eqref{eq:22192} gives
\begin{equation}\label{eq:22198}
\lim_{j\rightarrow{\infty}}{\sum_{i=1}^{m}\phi'_{+}((b_{i}-a_{i}^{T}x^{k_{j}})^{2})(b_{i}-a_{i}^{T}\tilde{x}^{k_{j}+1}){a_{i}}}=0.
\end{equation}
Next, observe from item (ii) that
\[
\begin{aligned}
& \lim_{j\to \infty}\phi'_+((b_i - a_i^Tx^{k_j})^2)(b_i - a_i^T\tilde x^{k_j+1})\\
  & = \lim_{j\to \infty}\phi'_+((b_i - a_i^Tx^{k_j})^2)(b_i - a_i^Tx^{k_j}) = \phi'_{+}((b_{i}-a_{i}^{T}x^{*})^{2})(b_{i}-a_{i}^{T}{x}^{*}).
\end{aligned}
\]
Combining the above display with \eqref{eq:22198}, we obtain
\begin{equation*}
{\sum_{i=1}^{m}\phi'_{+}((b_{i}-a_{i}^{T}x^{*})^{2})(b_{i}-a_{i}^{T}{x}^{*}){a_{i}}}=0.
\end{equation*}
This leads to a contradiction because Proposition~\ref{pro:17} asserts that the MFCQ holds for \eqref{eq:0.1} and we have from item (iii) that $\Phi((b - Ax^*)\circ (b - Ax^*)) = \sigma$. Therefore, the sequence $\{\tilde{\lambda}_{k}\}$ is bounded.

{\rm (v)}: Let $x^{*}$ be an arbitrary accumulation point of $\{x^{k}\}$ and let $\{x^{k_j}\}$ be a subsequence such that $\lim_{j\rightarrow{\infty}}x^{k_{j}+1}=x^{*}$. Then we see from Proposition~\ref{prop:221121}(ii) that
\begin{equation}\label{x-limit}
\lim_{j\to\infty}\tilde x^{k_j+1} = x^*.
\end{equation}
Moreover, since $\{x^k\}$ is bounded in view of Proposition~\ref{prop:221121}(i), by passing to a further subsequence if necessary, we assume that
\begin{equation}\label{x-limit2}
\lim_{j\to \infty}x^{k_j} = \bar x
\end{equation}
for some $\bar x\in \R^n$. Then both $x^*$ and $\bar x$ are accumulation points of $\{x^k\}$. Furthermore, since $\{\tilde \lambda_k\}$ is a nonnegative sequence and is bounded in view of item (iv), by passing to a further subsequence if necessary, we assume without loss of generality that $\lim_{j\rightarrow{\infty}}\tilde \lambda_{k_j}=\lambda_*$ for some $\lambda_*\ge 0$.
Then we see immediately from item (iii) that
\[
\Phi((b-Ax^*)\circ(b-Ax^*))\leq\sigma~\ \ \mbox{and}\ \ ~\lambda_*(\Phi((b-Ax^*)\circ(b-Ax^*))-\sigma)=0,
\]
which means that both \eqref{eq:25} and \eqref{eq:26} hold with $(x^{*},\lambda_*)$ (and hence $(x^*,\lambda_*/2)$) in place of $(x,\lambda)$.
Thus, to complete the proof, it remains to show that $(x^{*},\lambda_*/2)$ satisfies \eqref{eq:27} in place of $(x,\lambda)$.
To this end, we first observe from item (ii) that
\begin{equation*}
\begin{aligned}
\lim_{j\to \infty}\phi'_+((b_i - a_i^Tx^{k_j})^2)(b_i - a_i^T\tilde x^{k_j+1})= \lim_{j\to \infty}\phi'_+((b_i - a_i^Tx^{k_j})^2)(b_i - a_i^Tx^{k_j})\ \ \ \forall i.
\end{aligned}
\end{equation*}
Let $I := \{i:\; \phi'_+((b_i - a_i^T\bar x)^2) > 0\}$.
If $i\in I$, we deduce immediately from the above display that $a_i^Tx^* = a_i^T\bar x$. On the other hand, for those $i\notin I$, we claim that $\phi'_+((b_i - a_i^Tx^*)^2) = 0$. Suppose to the contrary that $\phi'_+((b_{i_0} - a_{i_0}^Tx^*)^2) > 0$ for some $i_0\notin I$. Then by the concavity of $\phi$, we know that $(b_i - a_i^T\bar x)^2$ is a maximizer of $\phi$ for any $i\notin I$ and hence
\[
\phi((b_{i_0} - a_{i_0}^Tx^*)^2) < \phi((b_{i_0} - a_{i_0}^T\bar x)^2) \ \ \ {\rm and}\ \ \ \phi((b_i - a_i^Tx^*)^2) \le \phi((b_i - a_i^T\bar x)^2)\ \ \forall i\notin I\cup\{i_0\}.
\]
Also, recall that $\phi((b_i - a_i^Tx^*)^2) = \phi((b_i - a_i^T\bar x)^2)$ for all $i\in I$. Thus, we have
\begin{eqnarray*}
&& \Phi((b - Ax^*)\circ (b - Ax^*)) = \sum_{i=1}^m \phi((b_i-a^T_ix^*)^2) \\
&& = \phi((b_{i_0} - a_{i_0}^Tx^*)^2) + \sum_{i\neq i_0, i\notin I} \phi((b_i-a^T_ix^*)^2) + \sum_{i\in I} \phi((b_i-a^T_ix^*)^2)\\
&& = \phi((b_{i_0} - a_{i_0}^Tx^*)^2) + \sum_{i\neq i_0, i\notin I} \phi((b_i-a^T_ix^*)^2) + \sum_{i\in I} \phi((b_i-a^T_i\bar x)^2)\\
&& \le \phi((b_{i_0} - a_{i_0}^Tx^*)^2) + \sum_{i\neq i_0, i\notin I} \phi((b_i-a^T_i\bar x)^2) + \sum_{i\in I} \phi((b_i-a^T_i\bar x)^2)\\
&& < \phi((b_{i_0} - a_{i_0}^T\bar x)^2) + \sum_{i\neq i_0} \phi((b_i-a^T_i\bar x)^2) = \Phi((b - A\bar x)\circ (b - A\bar x)).
\end{eqnarray*}
Since both $x^*$ and $\bar x$ are accumulation points of $\{x^k\}$, the above display contradicts item (iii). Thus, we have shown that,
\begin{equation}\label{keyrelationship}
  \begin{cases}
    a^T_ix^* = a^T_i\bar x & {\rm if}\ \phi'_+((b_i - a_i^T\bar x)^2) > 0,\\
    \phi'_+((b_i - a_i^T\bar x)^2) = \phi'_+((b_i - a_i^Tx^*)^2)=0 & {\rm otherwise}.
  \end{cases}
\end{equation}
Also, with $\tilde v^{k+1}$ defined as in \eqref{eq:221190}, one can notice from \eqref{eq:221214} and the boundedness of the sequence $\{A_{k_{j}}^{T}\}$ (thanks to Remark \ref{rem:221211}) that
\begin{equation}\label{haha5}
\lim_{j\rightarrow{\infty}}A_{k_{j}}^{T}\tilde{v}^{k_{j}+1}=0.
\end{equation}
Finally, we deduce from \eqref{eq:22180}, \eqref{eq:221190}, and the definitions of $\{w^k\}$ and $\{A_k\}$ in Algorithm~\ref{algorithm 1} that
\begin{equation*}
\begin{aligned}
&\xi^{k_{j}}\in{w^{k_{j}}\circ\partial{\|\tilde{x}^{k_{j}+1}\|_1}+\tilde{\lambda}_{k_{j}}A_{k_{j}}^{T}(A_{k_{j}}\tilde{x}^{k_{j}+1}-b^{k_{j}}-\tilde{v}^{k_{j}+1})}\\
&=\Psi'_{+}(|x^{k_{j}}|)\circ\partial{\|\tilde{x}^{k_{j}+1}\|_1}-\tilde{\lambda}_{k_{j}}
\left({\sum_{i=1}^{m}\phi'_{+}((b_{i}-a_{i}^{T}x^{k_{j}})^{2})(b_{i}-a_{i}^{T}\tilde{x}^{k_{j}+1}){a_{i}}}+
A_{k_{j}}^{T}\tilde{v}^{k_{j}+1}\right)\\
&=\Psi'_{+}(|\tilde x^{k_{j}+1}|+\delta_{k_j})\circ\partial{\|\tilde{x}^{k_{j}+1}\|_1}-\tilde{\lambda}_{k_{j}}
\left({\sum_{i=1}^{m}\phi'_{+}((b_{i}-a_{i}^{T}x^{k_{j}})^{2})(b_{i}-a_{i}^{T}\tilde{x}^{k_{j}+1}){a_{i}}}+
A_{k_{j}}^{T}\tilde{v}^{k_{j}+1}\right),
\end{aligned}
\end{equation*}
where $\delta_{k_j} := |x^{k_{j}}| - |\tilde x^{k_{j}+1}|$, with $\delta_{k_j}\to 0$ as $j\to \infty$ in view of Proposition~\ref{prop:221121}(iii).
Passing to the limit as $j\to \infty$ in the above display and noting \eqref{x-limit}, \eqref{x-limit2}, $\tilde\lambda_{k_j}\to \lambda_*$, \eqref{keyrelationship}, \eqref{haha5} and the closedness of subdifferential,
we conclude that \eqref{eq:27} holds with $(x^{*},\lambda_*/2)$ in place of $(x,\lambda)$. This completes the proof.
\end{proof}

\section{Subproblem solvers and termination criteria}\label{sec4}

Recall from Remark \ref{rem3.1} that the feasible set of \eqref{subproblem2} is nonempty and the Slater condition holds. Moreover, it holds that $w^k_i > 0$ for each $k$ and $i$ in view of Assumption~\ref{ass:1.1}. Thus, the optimal solution set of the subproblem \eqref{subproblem2} is nonempty for each $k$, and we have in view of \cite[Corollary 28.2.1, Theorem 28.3]{convexana1970} that a Lagrange multiplier exists for this problem. Consequently, the subproblem \eqref{subproblem2} can be approximately solved by a wide range of first-order methods; see, for example, \cite{BCG11,Beck17,Nes2006} and references therein.

In this section, we will discuss two specific methods for solving subproblem \eqref{subproblem2} for each $k$, namely the alternating direction method of multipliers (ADMM) and SPGL1. Moreover, we will investigate how a tuple $(\tilde x^{k+1},\tilde{u}^{k+1})$ satisfying the inexact criteria \eqref{eq:0077742}, \eqref{eq:0077842} and \eqref{eq:0077831} can be found by the subproblem solvers.

For notational convenience, for the rest of this section, we fix any $k \ge 0$ and define
\begin{equation}\label{Abarbbar}
  \bar{A}:=A_k,\ \bar{b}:=b^k,\ \bar{\sigma}:=\sqrt{\sigma_k},\ \bar v := v^k,\ \bar{w}:=w^k
\end{equation}
so that we can rewrite the corresponding subproblem \eqref{subproblem2} simply as
\begin{equation}\label{eq:22325}
  \begin{array}{rl}
    \min\limits_{x\in{{\mathbb{R}}^{n}}} & \|\bar{w}\circ{x}\|_{1}\\
    {\rm s.t.} & \|\bar{b}-\bar{A}x\|\leq{\bar{\sigma}}.
  \end{array}
\end{equation}

\subsection{Alternating direction method of multipliers}\label{sec4.1}

Alternating direction method of multipliers (ADMM) is a classical method for minimizing the sum of two proper closed convex functions whose variables are coupled only by a linear constraint; see \cite{BPCPE10,Fazel2013} and references therein for more discussions. This method is also the basis for the solver YALL1 \cite{YangZhang11} for problems of the form \eqref{eq:22325}. A basic form of ADMM applied to solving \eqref{eq:22325} can be described as follows: Initialize at some suitable $(\bar x^0,\bar u^0,\bar \lambda^0)$ and compute for $l = 0,1,2,\ldots$,
\begin{subnumcases}{\label{eq:221300}}
\bar x^{l+1}\in {\Argmin_{x\in \R^n}\left\{L_{\beta}(x,\bar u^l;\bar \lambda^l)+{\frac{1}{2}}(x - \bar x^l)^T(\tilde{\lambda}{I}-\beta{\bar{A}^{T}\bar{A}})(x - \bar x^l)\right\}},\label{eq:221277}\\
\bar u^{l+1}\in {\Argmin_{u\in \R^m}}\left\{L_{\beta}(\bar x^{l+1},u;\bar \lambda^l)\right\},\label{eq:221278}\\
\bar \lambda^{l+1}=\bar \lambda^{l}-\gamma\beta(\bar{A}\bar x^{l+1}-\bar{b}-\bar u^{l+1}),\label{eq:22540}
\end{subnumcases}
where $\beta>0$, $\gamma\in(0,\frac{1+\sqrt{5}}{2})$, $\tilde \lambda > 0$ satisfies $\tilde{\lambda}I-\beta{\bar{A}^{T}\bar{A}}\succeq 0$, and $L_{\beta}$ is the augmented Lagrangian function defined by
\[
L_{\beta}(x,u;\lambda)=\|\bar{w}\circ{x}\|_{1}+\delta_{\|\cdot\|\leq{\bar{\sigma}}}(u)- \lambda^T(\bar{A}x-\bar{b}-u)+\frac{\beta}{2}\|\bar{A}x-\bar{b}-u\|^{2}.
\]
Notice that the $\bar x$-update and $\bar u$-update in \eqref{eq:221277} and \eqref{eq:221278} have closed form solutions in terms of the soft-thresholding operator \cite{Teboulle2009} and the projection onto the Euclidean norm ball of radius $\bar \sigma > 0$, respectively. Moreover, in view of the positivity of $\bar w_i$ as guaranteed by Assumption~\ref{ass:1.1} and the Slater's condition discussed in Remark~\ref{rem3.1}, we deduce from \cite[Theorem~B.1]{Fazel2013} that the sequence $\{\bar x^l,\bar u^l,\bar \lambda^l\}$ is well-defined, $\{\bar x^l\}$ converges to an optimal solution $\bar{x}^*$ of \eqref{eq:22325}, and furthermore,
\begin{equation}\label{eq:283}
\|\bar \lambda^l-\bar \lambda^{l+1}\|\rightarrow{0}, \|\bar x^{l+1}-\bar x^l\|\rightarrow{0}, \|\bar u^{l+1}-\bar u^l\|\rightarrow{0}~\mbox{and}~\|\bar{w}\circ{\bar x^l}\|_{1}\rightarrow{\|\bar{w}\circ{\bar{x}^*}\|_{1}} \ \ {\rm as}\ \ l\to \infty.
\end{equation}

We now argue that the criteria \eqref{eq:0077742}, \eqref{eq:0077842} and \eqref{eq:0077831} can be achieved with $(\tilde x^{k+1},\tilde u^{k+1}) = (\bar x^{l+1},\bar u^{l+1})$ for some large enough $l$ (depending on $k$). To this end, observe from the optimality conditions of the minimization problems \eqref{eq:221277} and \eqref{eq:221278} that
\begin{numcases}{}
0\in{\bar{w}\circ{\partial\|\bar x^{l+1}\|_1}}-\bar{A}^{T}\bar \lambda^l+\beta{\bar{A}^{T}}(\bar{A}\bar x^{l+1}-\bar{b}-\bar u^l)+(\tilde{\lambda}{I}-\beta{\bar{A}^{T}\bar{A}})(\bar x^{l+1}-\bar x^l),\nonumber\\
0\in{N_{\|\cdot\|\leq{\bar{\sigma}}}(\bar u^{l+1})+\bar \lambda^l-\beta{(\bar{A}\bar x^{l+1}-\bar{b}-\bar u^{l+1})}},\nonumber\\
\bar \lambda^{l+1}-\bar \lambda^l=-\gamma\beta{(\bar{A}\bar x^{l+1}-\bar{b}-\bar u^{l+1})}.\nonumber
\end{numcases}
This implies that
\begin{numcases}{}
-{\beta}\bar{A}^{T}(\bar u^{l+1}-\bar u^l)-(\tilde{\lambda}{I}-\beta{\bar{A}^{T}\bar{A}})(\bar x^{l+1}-\bar x^l)\in{\bar{w}\circ{\partial\|\bar x^{l+1}\|_1}+\bar{A}^{T}N_{\|\cdot\|\leq{\bar{\sigma}}}(\bar u^{l+1})},\label{eq:282}\\
\bar \lambda^{l+1}-\bar \lambda^l=-\gamma\beta{(\bar{A}\bar x^{l+1}-\bar{b}-\bar u^{l+1})}.\label{eq:281}
\end{numcases}
Combining \eqref{eq:283}, \eqref{eq:282} and \eqref{eq:281} and recalling that $\epsilon_k > 0$, we deduce that for all sufficiently large $l$,
\begin{align}
&\mathrm{dist}(0,\bar{w}\circ{\partial{\|\bar x^{l+1}\|}}+\bar{A}^{T}N_{{\|\cdot\|}\leq{\bar{\sigma}}}(\bar u^{l+1}))\nonumber\\
&\leq{\|-{\beta}\bar{A}^{T}(\bar u^{l+1}-\bar u^l)-(\tilde{\lambda}{I}-\beta{\bar{A}^{T}\bar{A}})(\bar x^{l+1}-\bar x^l)\|}\leq{\epsilon_{k}}\label{eq:225161}
\end{align}
and
\begin{equation}\label{eq:225162}
\|\bar{A}\bar x^{l+1}-\bar{b}-\bar u^{l+1}\|=(\gamma\beta)^{-1}\|\bar \lambda^l-\bar \lambda^{l+1}\|\leq{\epsilon_{k}},
\end{equation}
which means that the pair $(\bar x^{l+1},\bar u^{l+1})$ satisfies the criteria \eqref{eq:0077742} and \eqref{eq:0077842} (in place of $(\tilde x^{k+1},\tilde u^{k+1})$) for all sufficiently large $l$.
As for criterion \eqref{eq:0077831}, note that $\mu_k > 0$ and we have $\lim_{l\rightarrow{\infty}}\|\bar{w}\circ{P_{k}(\bar x^{l+1})}\|_{1}={\|\bar{w}\circ{\bar{x}^*}\|_{1}}$ in view of \eqref{eq:283}, the fact that $\{\bar x^l\}$ converges to an optimal solution $\bar{x}^*$ of \eqref{eq:22325}, and the definition of $P_k$ in \eqref{eq:225164}. Thus, we also have for all large $l$ that
\[
\|\bar{w}\circ{P_{k}(\bar x^{l+1})}\|_{1} < {\|\bar{w}\circ{\bar{x}^*}\|_{1}} + \mu_k \le {\|\bar{w}\circ x^k\|_{1}} + \mu_k,
\]
where the second inequality holds
because $x^k$ is feasible for \eqref{eq:22325} and $\bar x^*$ is an optimal solution of \eqref{eq:22325}. Hence, the criterion \eqref{eq:0077831} is also satisfied by $(\bar x^{l+1},\bar u^{l+1})$  (in place of $(\tilde x^{k+1},\tilde u^{k+1})$) for all large enough $l$.

\subsection{Spectral projected-gradient algorithm for $L_1$ minimization}\label{sec4.2}

SPGL1 \cite{Michael2008} is a standard solver for problems of the form \eqref{eq:22325}. The basic idea behind SPGL1 is a root-finding procedure for the following nonsmooth equation
\begin{equation}\label{varphi}
\varphi(\tau) = \bar\sigma,\ \ \ \ {\rm where}\ \varphi(\tau):=\displaystyle \inf_{\|\bar{w}\circ{x}\|_1\leq{\tau}} \|\bar{A}x-\bar{b}\|.
\end{equation}
Notice that $\|\bar b\| > \bar \sigma$ (because $0$ is not feasible for \eqref{eq:0.1} and hence is also not feasible for \eqref{eq:22325}), $\bar \sigma > 0$ (see Lemma~\ref{lem:prep}) and $\varphi(\|\bar w\circ (A^\dagger b)\|_1) = 0$.
One can seen from \cite[Section 2]{Michael2008} that the optimal {\em value} of \eqref{eq:22325}, denoted by $\tau_{\bar{\sigma}}$, is the {\em unique} solution of the above equation. In each iteration of the SPGL1, the function $\varphi$ and its gradient are approximately evaluated at the currently available $\tau$ value, and the Newton's method (with inexact gradient and function value) is applied to solving the above equation. Specifically, starting with a $\tau_0 = 0$, for each $l = 0,1,2,\ldots$, one computes
\begin{equation}
\tau_{l+1}=\tau_{l}+\frac{\bar{\sigma}-\tilde{\varphi}_l(\tau_{l})}{\tilde{\varphi}_l'(\tau_{l})},\label{eq:223301}
\end{equation}
where $\tilde{\varphi}_l(\tau_{l})$ and $\tilde{\varphi}_l'(\tau_{l})$ are approximations of $\varphi(\tau_{l})$ and ${\varphi}'(\tau_{l})$, respectively, that are constructed to satisfy technical assumptions such as the affine minorant oracles in \cite[Algorithm 2.2]{Michael2019} (see also \cite[Section~3.1]{Michael2008}); specifically, their constructions are based on a ``sufficiently accurate" approximate solution of the following optimization problem
\begin{equation}\label{eq:223243}
  \begin{array}{rl}
\min\limits_{\|\bar{w}\circ{x}\|_1\leq{\tau_{l}}}\|\bar{b}-\bar{A}x\|,
  \end{array}
\end{equation}
which is solved approximately by the spectral projected-gradient (SPG) method described in \cite[Algorithm 1]{Michael2008}.
From \cite[Theorem 2.3]{Michael2019} (see also \cite[Section~3]{Michael2008}) and recalling $\|\bar b\|> \bar \sigma$, we know that the sequence $\{\tau_l\}$ generated from \eqref{eq:223301} converges to $\tau_{\bar{\sigma}}$, the unique solution of \eqref{varphi}.
Moreover, since $\bar{\sigma}>0$ in view of Lemma \ref{lem:prep}, we can deduce from \cite[Theorem 2.1]{Michael2008} (see also the discussion at the beginning of \cite[Section~3]{Michael2008}) that
\begin{equation}\label{varphi0}
\lim_{l\rightarrow{\infty}}\varphi(\tau_l)=\bar\sigma\ \ {\rm and}\ \ \lim_{l\rightarrow{\infty}}\tau_l=\tau_{\bar{\sigma}}<\tau_{BP},
\end{equation}
where $\tau_{BP} > 0$ is the smallest positive root of $\varphi$. Furthermore, we also have $\tau_{\bar \sigma} > 0$ because $\varphi(0) = \|\bar b\| > \bar \sigma$. The above observations show that there exists $l_0 > 0$ such that
\begin{equation}\label{varphi2}
\varphi(\tau_l) > 0\ {\rm and} \ \tau_l > 0\ \ \forall l\ge l_0.
\end{equation}

We next argue that the criteria \eqref{eq:0077742}, \eqref{eq:0077842} and \eqref{eq:0077831} can be achieved when $l$ is sufficiently large
based on sufficiently accurate solutions of \eqref{eq:223243}; note that we may need more accurate solutions than those used for constructing $\tilde\varphi_l$ and $\tilde\varphi_l'$ (and hence $\tau_{l+1}$) above.
To this end, we first recall some facts concerning the SPG method used for solving \eqref{eq:223243}. Recall from
\cite[Algorithm 1]{Michael2008} that the SPG method can be described as follows: For each $l$, we initialize $\bar x^{l,0}$ suitably and compute, for $t=0,1,2,\ldots$,
\begin{equation}\label{eq:2233033}
\bar x^{l,t+1}\in\Argmin_{\|\bar{w}\circ{x}\|_{1}\leq{\tau_{l}}}\frac12\|x-(\bar x^{l,t}+\alpha_{l,t}{\bar{A}^T}(\bar{b}-\bar{A}\bar x^{l,t}))\|^2,
\end{equation}
where $\alpha_{l,t} \in (0,\alpha_{\max}]$ is found by backtracking starting from ``uniformly positive" initial stepsizes to satisfy an Armijo-type nonmonotone linesearch condition, with $\alpha_{\max}$ being a fixed positive parameter. It is standard to show that $\inf_{t,l}\alpha_{l,t}\ge \alpha_{\min}$ for some positive constant $\alpha_{\min}$; see, for example, \cite[Lemma~5(b)]{TsengYun09}. Based on the existence of such a lower bound on $\{\alpha_{l,t}\}$, one can further show that for each $l$,
\begin{equation}\label{convergence}
\lim_{t\rightarrow{\infty}}\|\bar x^{l,t+1}-\bar x^{l,t}\|=0;
\end{equation}
see, for example, \cite[Lemma~3.8]{LuZhang12}. Moreover, one has from \cite[Theorem 2.4]{Bergin2000} that any accumulation point of $\{\bar x^{l,t}\}$ (as a sequence in $t$) is a global minimizer of \eqref{eq:223243}. In view of this and \eqref{varphi}, we have for each $l$,
\begin{equation}\label{convergence2}
\lim_{t\rightarrow{\infty}}\|\bar A \bar x^{l,t+1} - \bar b\| = \varphi(\tau_l).
\end{equation}

Now, from the first-order optimality condition of \eqref{eq:2233033}, we deduce for each $l\ge l_0$ (recall that $l_0$ is given in \eqref{varphi2}) and each $t\ge 0$ that
\begin{equation*}
0\in{\bar x^{l,t+1}-\bar x^{l,t}+\alpha_{l,t}{\bar{A}^T}(\bar{A}\bar x^{l,t}-\bar{b})+N_{C_l}(\bar x^{l,t+1})},
\end{equation*}
where $C_l := \{x:\; \|\bar{w}\circ{x}\|_{1}\leq{\tau_{l}}\}$.
Since $\alpha_{l,t}\ge \alpha_{\min} > 0$, we further have
\begin{equation*}
0\in{\frac{1}{\alpha_{l,t}}(\bar x^{l,t+1}-\bar x^{l,t})+{\bar{A}^T}(\bar{A}\bar x^{l,t}-\bar{b})+N_{C_l}(\bar x^{l,t+1})}.
\end{equation*}
Since $\tau_l > 0$ in view of \eqref{varphi2} and $l\ge l_0$, the above display together with Theorem~1.3.5 in \cite[Section~D]{HiLe01} further implies that there exists a $\lambda_{l,t}\geq{0}$ such that
\begin{equation}\label{lambda}
0\in{\frac{1}{\alpha_{l,t}}(\bar x^{l,t+1}-\bar x^{l,t})+{\bar{A}^T}(\bar{A}\bar x^{l,t}-\bar{b})+\lambda_{l,t}\bar{w}\circ{\partial\|\bar x^{l,t+1}\|_1}}\ \ {\rm and}\ \ \lambda_{l,t}(\|\bar w\circ\bar x^{l,t+1}\|_1-\tau_l)=0.
\end{equation}
Using the first relation in \eqref{lambda} and recalling that $\alpha_{l,t}\in [\alpha_{\min},\alpha_{\max}]$, we deduce that for all $t$,
\begin{equation*}
\begin{aligned}
&{\rm dist}(0,\lambda_{l,t}\bar{w}\circ{\partial\|\bar x^{l,t+1}\|_1} + \bar{A}^{T}(\bar{A}\bar x^{l,t+1} - \bar{b}))\\
&\le \|\alpha_{l,t}^{-1}(\bar x^{l,t}-\bar x^{l,t+1}) + {\bar{A}^T}\bar{A}(\bar x^{l,t+1}-\bar x^{l,t})\| \le \widehat M\|\bar x^{l,t+1} - \bar x^{l,t}\|,
\end{aligned}
\end{equation*}
where $\widehat M := \lambda_{\max}(\bar A^T\bar A) + \alpha_{\min}^{-1}$.
The above display together with \eqref{convergence} implies that for each $l\ge l_0$,
\begin{equation}\label{eq:223261}
\lim_{t\rightarrow{\infty}}{\rm dist}(0,\lambda_{l,t}\bar{w}\circ{\partial\|\bar x^{l,t+1}\|_1} + \bar{A}^{T}(\bar{A}\bar x^{l,t+1} - \bar{b}))=0.
\end{equation}

We are now ready to show that the criteria \eqref{eq:0077742}, \eqref{eq:0077842} and \eqref{eq:0077831} can be achieved.
\begin{itemize}
  \item {\bf On criterion \eqref{eq:0077742}}: Let $\bar u^{l,t+1}$ denote a projection of $\bar A\bar x^{l,t+1}-\bar b$ onto the sphere $\{x:\;\|x\| = \bar\sigma\}$. Then from \eqref{varphi0} and \eqref{convergence2}, we deduce that there exists $l_1$ such that for all $l\ge l_1$, there exists $s_{l}$ such that
\begin{equation*}
\|\bar A \bar x^{l,t+1} - \bar b - \bar u^{l,t+1}\| < \epsilon_k\ \ \ \ \ \forall t\ge s_{l}.
\end{equation*}
This shows that \eqref{eq:0077742} is satisfied by $(\bar x^{l,t+1},\bar u^{l,t+1})$ (in place of $(\tilde x^{k+1},\tilde u^{k+1})$) for all large $l$ and the correspondingly large $t$.
\item {\bf On criterion \eqref{eq:0077842}}: We start by noting from \eqref{varphi0}, \eqref{convergence2}, \eqref{eq:223261} and the definition of $u^{l,t+1}$ that
\begin{equation}\label{limit333}
  \lim_{l\to\infty}\lim_{t\rightarrow{\infty}}{\rm dist}(0,\lambda_{l,t}\bar{w}\circ{\partial\|\bar x^{l,t+1}\|_1} + \bar{A}^{T}\bar u^{l,t+1})=0.
\end{equation}
We next derive a {\em uniform} lower bound on the magnitude of $\{\lambda_{l,t}\}$ for all large $l$ and all correspondingly large $t$.
To this end, given any $l\ge l_0$, upon rearranging terms in the first relation in \eqref{lambda} and noting that $\partial\|\cdot\|_1$ is contained in the infinity norm ball of radius 1, we obtain
the following lower bound for $|\lambda_{l,t}|$ whenever $l\ge l_0$:
\begin{equation*}
|\lambda_{l,t}| \ge \frac{1}{\max_{1\le i\le n}\bar w_i}\left\|\frac{1}{\alpha_{l,t}}(\bar x^{l,t+1}-\bar x^{l,t})+{\bar{A}^T}(\bar{A}\bar x^{l,t}-\bar{b})\right\|_\infty.
\end{equation*}
Hence, in view of \eqref{convergence} and recalling that $\alpha_{l,t}\in [\alpha_{\min},\alpha_{\max}]$, we have
\begin{align*}
\Lambda:= \liminf_{l\to\infty}\liminf_{t\rightarrow{\infty}}|\lambda_{l,t}| &\ge \frac{1}{\max_{1\le i\le n}\bar w_i}\liminf_{l\to\infty}\liminf_{t\rightarrow{\infty}}\left\|{\bar{A}^T}(\bar{A}\bar x^{l,t}-\bar{b})\right\|_\infty\\
&\ge \frac{1}{n\max_{1\le i\le n}\bar w_i}\liminf_{l\to\infty}\liminf_{t\rightarrow{\infty}}\left\|{\bar{A}^T}(\bar{A}\bar x^{l,t}-\bar{b})\right\|\\
&\overset{\rm (a)}= \frac{1}{n\max_{1\le i\le n}\bar w_i}\liminf_{l\to\infty}\liminf_{t\rightarrow{\infty}}\left\|{A^T}{{\rm Diag}(\bar v)}(\bar{A}\bar x^{l,t}-\bar{b})\right\|\\
&\ge \frac{\sqrt{\lambda_{\min}(AA^T)}}{n\max_{1\le i\le n}\bar w_i}\liminf_{l\to\infty}\liminf_{t\rightarrow{\infty}}\left\|{{\rm Diag}(\bar v)}(\bar{A}\bar x^{l,t}-\bar{b})\right\|\\
&\overset{\rm (b)}\ge \frac{\sqrt{\lambda_{\min}(AA^T)}}{n\max_{1\le i\le n}\bar w_i}\cdot\min_{\bar v_i >0}\bar v_i\cdot\liminf_{l\to\infty}\liminf_{t\rightarrow{\infty}}\left\|\bar{A}\bar x^{l,t}-\bar{b}\right\| > 0,
\end{align*}
where (a) holds because $\bar A = {\rm Diag}(\bar v)A$ (see \eqref{Abarbbar} and Step 1 of $\IR$), (b) holds because $\bar A = {\rm Diag}(\bar v)A$ and $\bar b =  {\rm Diag}(\bar v)b$, and the last strict inequality follows from the facts that $A$ has full row rank (Assumption~\ref{ass:1.1}(iii)) and that $\lim_{l\to\infty}\lim_{t\rightarrow{\infty}}\|\bar A \bar x^{l,t} - \bar b\| = \bar\sigma > 0$, thanks to \eqref{varphi0}, \eqref{convergence} and \eqref{convergence2}. Thus, there exists $l_2\ge l_0$ such that for all $l\ge l_2$, there exists $\gamma_l$ such that whenever $t\ge \gamma_l$, we have
\begin{equation}\label{positivity}
\lambda_{l,t}\ge \Lambda/2 > 0.
\end{equation}
Equipped with this lower bound, we can now invoke \eqref{limit333} to deduce the existence of $l_3\ge l_2$ such that for all $l\ge l_3$, there exists $\zeta_l \ge \gamma_l$ such that whenever $t\ge \zeta_l$, we have
\[
{\rm dist}(0,\lambda_{l,t}\bar{w}\circ{\partial\|\bar x^{l,t+1}\|_1} + \bar{A}^{T}\bar u^{l,t+1}) \le \Lambda \epsilon_k/2 \le \lambda_{l,t}\epsilon_k.
\]
This shows that
\[
{\rm dist}(0,\bar{w}\circ{\partial\|\bar x^{l,t+1}\|_1} + \lambda_{l,t}^{-1}\bar{A}^{T}\bar u^{l,t+1}) \le \epsilon_k.
\]
The above display together with the normal cone formula in Theorem~1.3.5 of \cite[Section~D]{HiLe01} and the fact that $\|\bar u^{l,t+1}\| = \bar \sigma > 0$ shows that \eqref{eq:0077842} is satisfied by $(\bar x^{l,t+1},\bar u^{l,t+1})$ (in place of $(\tilde x^{k+1},\tilde u^{k+1})$) for all $l\ge l_3$ and $t\ge \zeta_l$.
\item {\bf On criterion \eqref{eq:0077831}}: We note from \eqref{positivity}, \eqref{varphi0} and \eqref{lambda} that
\[
\lim_{l\rightarrow{\infty}}\lim_{t\to\infty}\|\bar{w}\circ \bar x^{l,t+1}\|_{1}=\tau_{\bar \sigma}.
\]
Moreover, we see from \eqref{varphi0} and \eqref{convergence2} that
\[
\lim_{l\rightarrow{\infty}}\lim_{t\to\infty}\|\bar A\bar x^{l,t+1}-\bar b\| = \bar \sigma.
\]
Using the above two displays, the definition of $P_k$ in \eqref{eq:225164} and the positivity of $\mu_k$, we deduce that the criterion \eqref{eq:0077831} is also satisfied by $(\bar x^{l,t+1},\bar u^{l,t+1})$ (in place of $(\tilde x^{k+1},\tilde u^{k+1})$) for all sufficiently large $l$ and all correspondingly large $t$.
\end{itemize}

\begin{remark}
  For SPGL1, as discussed above, theoretically, a tuple satisfying the criteria \eqref{eq:0077742}, \eqref{eq:0077842} and \eqref{eq:0077831} exists. However, for each subproblem instance, it can be difficult to estimate the thresholds for $l$ and $t$, and other parameters of the solver might need to be carefully adjusted as well to obtain such a tuple efficiently. In our numerical experiments in the next section, we will simply treat SPGL1 as a black-box solver with its default parameter settings, and invoke the adjustment discussed in Remark~\ref{rem3.2}. We will give more details in the next section.
\end{remark}

\section{Numerical experiments}\label{sec5}

In this section, we perform numerical experiments to study the behavior of $\IR$. Specifically, we consider the following constrained optimization problem
\begin{equation}\label{eq:5.1}
  \begin{array}{rl}
    \min\limits_{x\in{{\mathbb{R}}^{n}}} &\sum_{i=1}^{n}\log(1+|x_i|/{\epsilon})\\
    {\rm s.t.} & {\sum_{i=1}^{m}\log(1+(b_i-a^T_ix)^2/{\delta}^2)\leq{\sigma}},\\
  \end{array}
\end{equation}
where $m\ll{n}$, $\epsilon>0$, $\delta>0$, $b\in \R^m$, $\sigma\in (0,\sum_{i=1}^{m}\log(1+b_i^2/{\delta^2}))$, and
$\{a_1,\ldots,a_m\}$ is linearly independent. The above model is a special case of \eqref{eq:0.1} with $A$ being a matrix with its $i$th-row being $a^T_i$ for all $i$, and $\psi$ and $\phi$ being the log-penalty and Cauchy loss functions, respectively; recall that these $\psi$ and $\phi$ satisfy Assumption~\ref{ass:1.1} as mentioned in Section~\ref{sec1}). Notice that Assumption~\ref{ass:0.2} also holds because $\bar \phi = \infty$ in this case. Then we see from Theorem~\ref{the3.1} that every accumulation point of the sequence $\{x^k\}$ generated by $\IR$ is a stationary point of \eqref{eq:5.1}.

As a benchmark, we also consider solving \eqref{eq:5.1} using the SCP$_{\rm ls}$ in \cite{YuLP21}. Specifically, we rewrite \eqref{eq:5.1} into the following form:
\begin{equation}\label{decomposition}
\min_{x\in{\R^n}}F(x):=l_\psi{\sum_{i=1}^{m}|x_i|}-\left(\sum_{i=1}^{m}[l_\psi|x_i|-\psi(|x_i|)]\right)+\delta_{\sum_{i=1}^{m}\phi((b_i-a_i^T\cdot)^2)\le\sigma}(x),
\end{equation}
where $l_\psi:=\lim_{t\downarrow 0}\psi'(t)$, and we recall that $\psi$ and $\phi$ are the log-penalty and Cauchy loss function as in \eqref{eq:5.1} respectively. Writing $P_1(x):=l_\psi{\sum_{i=1}^{m}|x_i|}$, $P_2(x):=\sum_{i=1}^{m}[l_\psi|x_i|-\psi(|x_i|)]$ and $g(x):=\sum_{i=1}^{m}\phi((b_i-a_i^Tx)^2)-\sigma$ for notational simplicity, one can see that $P_1$ and $P_2$ are convex and continuous on $\R^n$, and $g$ has Lipschitz continuous gradient. Moreover, the level-boundedness of $\psi$ implies that of $F$, and we also have $\{x:g(x)\leq{0}\}\neq\emptyset$ since $A$ has full row rank and $\sigma\in (0,\sum_{i=1}^{m}\log(1+b_i^2/{\delta^2}))$. Finally, since $\bar \phi =\infty$, we deduce from Proposition \ref{pro:17} that the MFCQ holds. Thus, one can apply the SCP$_{\rm ls}$ in \cite{YuLP21} with $f = 0$ and the $P_1$, $P_2$ and $g$ defined above, and \cite[Theorem~3.2]{YuLP21} guarantees that any accumulation point of the sequence $\{x^k\}$ generated by SCP$_{\rm ls}$ is stationary in the sense of \cite[Definition~2.2]{YuLP21}.

In our experiments below, we solve \eqref{eq:5.1} by SCP$_{\rm ls}$ and a version of $\IR$ with its subproblems approximately solved by ADMM (see Section~\ref{sec4.1} for discussions of ADMM). We refer to this latter algorithm as $\IR$$_\mathrm{ADMM}$. We also consider a variant of $\IR$ where the subproblems are solved by SPGL1 with default parameters (see Section~\ref{sec4.2} for discussions of SPGL1) such that \eqref{eq:0077742}, \eqref{eq:0077842} and \eqref{eq:0077831} may not be satisfied. We refer to this algorithm as $v\IR$$_\mathrm{SPGL1}$. Our codes are written in MATLAB and all numerical experiments are conducted in MATLAB 2019b on a 64-bit PC with  an Intel Core i7-6700 CPU (3.40GHz) and 32GB of RAM.


\paragraph{Parameter settings:} For $\IR$$_\mathrm{ADMM}$ and $v\IR$$_\mathrm{SPGL1}$, we set $L=\lambda_{\max}(AA^{T})$.\footnote{In our numerical experiment, $L=\lambda_{\max}(AA^{T})$ is computed using the MATLAB commands: {\sf if m $>$ 2000
        opts.issym = 1;
        L = eigs(A*A',1,'LM',opts);
    else
        L = norm(A*A');
    end}. }
We initialize both algorithms at $x_{\rm feas}:=A^{\dag}{b}$, where $A^{\dag}{b}$ is computed via the MATLAB commands:
\begin{verbatim}
                       [Q,R] = qr(A',0); xfeas = Q*(R'\b);
\end{verbatim}
and terminate them when the stopping criterion $\frac{\|x^{k+1}-x^k\|}{\max\{\|x^k\|,1\}}\leq{{10^{-4}}}$ is satisfied.\footnote{It is indeed not known whether $\lim_{k\to\infty}\|x^{k+1}-x^k\|=0$ for the sequence $\{x^k\}$ generated by these algorithms. We use this criterion as a heuristic and it appears to work well.}
\begin{itemize}
  \item $\IR$$_\mathrm{ADMM}$: We let $\tau_k=\max\{5^{-k-1},10^{-8}\}$ and $\mu_k=\max\{1.2^{-k-1},10^{-8}\}$ for $\IR$$_\mathrm{ADMM}$. Its subproblem solver is based on ADMM in Section~\ref{sec4.1}. Using the notation in Section~\ref{sec4.1}, we set $\gamma=\frac{0.99(1+\sqrt{5})}{2}$, $\beta=\bar{L}^{-\frac12}$ and $\tilde{\lambda}=\bar{L}\beta$, where $\bar{L}:=\max\limits_{i}\{\phi'_+((b_i-a^T_ix^k)^2)\}L$. With this choice of $\bar L$, one can see from the definition of $\bar{A}$ in \eqref{Abarbbar} that $\bar{L}\geq{\lambda_{\max}(\bar{A}^T\bar{A})}$, and thus $\tilde{\lambda}I-\beta\bar{A}^T\bar{A}=\beta(\bar{L}I-\bar{A}^T\bar{A})\succeq{0}$.

  We initialize the ADMM at $\bar x^0=0$, $\bar u^0=0$, $\bar\lambda^0=0$ when $k = 0$, and at each (outer) iteration $k$, we warm-start it using the $(\bar x^l,\bar u^l,\bar \lambda^l)$ obtained from the previous (outer) iteration. As for termination, following the discussion in Section~\ref{sec4.1} (see especially \eqref{eq:225161} and \eqref{eq:225162}), for each $k\ge 0$, we terminate the ADMM when all of the following three conditions are satisfied:
      \begin{align*}
      {\|-{\beta}\bar{A}^{T}(\bar u^{l+1}-\bar u^l)-(\tilde{\lambda}{I}-\beta{\bar{A}^{T}\bar{A}})(\bar x^{l+1}-\bar x^l)\|}\leq\min\{{\bar{\epsilon}_k},\tau_k\bar \Gamma_l\},
      \end{align*}
     \[
      {\|\bar \lambda^l-\bar \lambda^{l+1}\|}\leq\gamma\beta\min\{\bar{\epsilon}_k,\tau_k (\|\bar\lambda^{l+1}\|+1)\},
      \]
     and
  \[
  \|\bar{w}\circ{P_k(\bar{x}^{l+1})}\|_1\leq{\|\bar{w}\circ{x^k}\|_1+\mu_k},
  \]
  where $\bar{\epsilon}_k=\min\{\sigma_k,\sqrt{\sigma_k}\}$ and $\bar\Gamma_l := \beta\|\bar A^T \bar u^{l+1} + (\bar L {I}-{\bar{A}^{T}\bar{A}})\bar x^{l+1}\|+1$. We then set $(\tilde x^{k+1},\tilde u^{k+1}) := (\bar x^{l+1},\bar u^{l+1})$.

 \item $v\IR$$_\mathrm{SPGL1}$: We call the \verb+spg-bpdn+ function from the source code of SPGL1 for solving the subproblems.\footnote{The codes were downloaded from https://github.com/mpf/spgl1.} We use the default parameters and termination criteria of this solver. As the point $\breve x^{k+1}$ generated by SPGL1 may violate the constraint slightly, we apply the function $P_k(\cdot)$ in \eqref{eq:225164} and set $x^{k+1} = P_k(\breve x^{k+1})$. From Remark~\ref{rem3.2}, this variant of $\IR$ (where the criteria \eqref{eq:0077742}, \eqref{eq:0077842} and \eqref{eq:0077831} may be violated) is well defined.\footnote{Nevertheless, there is no guarantee that the $\{x^k\}$ thus generated will cluster at a stationary point of \eqref{eq:5.1}. We include this version in our experiment as a demonstration of how our framework can be used when only a black-box subproblem solver is available.}
\end{itemize}

For SCP$_{\rm ls}$, we use a terminating tolerance of $10^{-5}$ instead of the $10^{-8}$ used in \cite[Section~5.3]{YuLP21}, and the other parameter settings and subproblem solver are the same as described in \cite[Section~5.3]{YuLP21}.

\paragraph{Test instances:} We consider randomly generated problems in our experiments below. We first generate an $m\times{n}$ matrix $A$ with i.i.d standard Gaussian entries. We then randomly choose a support set $S$ of size $s$ from $\{1,2,\ldots,n\}$ and generate an $s$-sparse vector $x_{\rm orig}\in \R^n$ with i.i.d. standard Gaussian entries on $S$. We further set $b=Ax_{\rm orig}+0.01\eta$, where $\eta\in{\mathbb{R}}^{m}$ has i.i.d. standard Cauchy entries.
Finally, we let $\sigma=1.2\sum_{i=1}^{m}\log(1 +(0.01\eta_i)^2/\delta^2)$, with $\delta = 0.05$.

\paragraph{Numerical results:} In our numerical experiments, we set $\epsilon = 0.1$ in \eqref{eq:5.1} and choose $(m,n,s)=(540i,2560i,80i)$ with $i=\{2,4,6,8,10\}$. For each $i$, we generate 30 random instances as described above and report in Table~\ref{Table2} the average value of $L$ ($\mathrm{L}$), the average CPU time (in seconds) for generating $L$ (Time$_\mathrm{L}$) and $A^{\dag}{b}$ (Time$_\mathrm{slater}$), and the average CPU time (in seconds) for computing $\mathrm{QR}$ decomposition of $A^T$ (Time$_\mathrm{QR}$). Then, in Table~\ref{Table3}, we present the computational results for $\IR$$_\mathrm{ADMM}$, $v\IR$$_\mathrm{SPGL1}$ and SCP$_{\rm ls}$. These computational results include the average total number of (inner) iterations (Iter$_\mathrm{s}$),\footnote{For fair comparison, we report the total number of inner iterations for $\IR$$_\mathrm{ADMM}$ and $v\IR$$_\mathrm{SPGL1}$, i.e., the total number of iterations used by the subproblem solvers to solve the subproblems.} the average CPU time in seconds (CPU$_\mathrm{s}$) and the average recovery error (RecErr$_\mathrm{s}$) among the successful instances; here, we declare the random instance to be successfully solved if ${\rm recovery~error} := \frac{\|\zeta^k-x_{\rm orig}\|}{\max\{\|x_{\rm orig}\|,1\}}\leq{0.01}$, where $k$ is the terminating iteration and
\[
\zeta^k := \begin{cases}
  \tilde x^k & \mbox{ for $\IR$$_\mathrm{ADMM}$ and $v\IR$$_\mathrm{SPGL1}$},\\
  x^k & \mbox{ for SCP$_{\rm ls}$.}
\end{cases}
\]
We also report the corresponding data for failed cases, which are denoted by Iter$_\mathrm{f}$, CPU$_\mathrm{f}$ and RecErr$_\mathrm{f}$, respectively. Furthermore, we list the success rate (Success), and the maximum and minimum values of the residual (Res$_\mathrm{max}$ and Res$_\mathrm{min}$) at termination, where
\[
{\rm residual} := \frac{1}{\sigma}\left(\sum_{i=1}^m \log(1 + (b_i - a_i^T \zeta^k)^2/\delta^2) - \sigma\right).
\]
One can see from the computational results in Table~\ref{Table3} that $\IR$$_\mathrm{ADMM}$ and $v\IR$$_\mathrm{SPGL1}$ can return solutions of \eqref{eq:5.1} with better recovery errors, and are always faster than SCP$_{\rm ls}$.

\begin{table}[h]
\begin{center}
\caption{The value of $L$, the CPU time needed for generating $L$ and $A^{\dag}{b}$ and for computing $\mathrm{QR}$ decomposition of $A^T$}\label{Table2}
~~~\\
\begin{tabular}{ccccc}
\hline
\multicolumn{1}{c}{$i$} & \multicolumn{1}{c}{\small{$\mathrm{L}$}} & \multicolumn{1}{c}{\small{Time$_\mathrm{L}$}}
& \multicolumn{1}{c}{\small{Time$_\mathrm{QR}$}} & \multicolumn{1}{c}{\small{Time$_\mathrm{slater}$}}\\
\bottomrule
2  & 1.08e+04  &    0.2  &    0.3  &    0.0   \\
4  & 2.18e+04  &    0.6  &    1.6  &    0.0   \\
6  & 3.27e+04  &    1.7  &    5.0  &    0.0   \\
8  & 4.36e+04  &    3.7  &   12.0  &    0.1   \\
10 & 5.44e+04  &    6.6  &   20.8  &    0.1   \\
\hline
\end{tabular}
\end{center}
\end{table}

\begin{table}[h]
\begin{center}
\caption{Comparison of the performance of $\IR$$_\mathrm{ADMM}$, $v\IR$$_\mathrm{SPGL1}$ and SCP$_{ls}$}\label{Table3}
~~~\\
\begin{tabular}{cccccccccc}
\hline
\multicolumn{1}{c}{$i$} & \multicolumn{9}{l}{\small{$\IR$$_\mathrm{ADMM}$}} \\ \cline{2-10}
~ & \small{Success(\%)} & \small{Iter$_\mathrm{s}$} & \small{Iter$_\mathrm{f}$} & \small{CPU$_\mathrm{s}$} & \small{CPU$_\mathrm{f}$} & \small{RecErr$_\mathrm{s}$} &  \small{RecErr$_\mathrm{f}$} & \small{Res$_\mathrm{min}$} & \small{Res$_\mathrm{max}$}\\
\bottomrule
2  & 100  &          4408  &           -  &   14.4  &    -  & 2.0e-03  & -  & -3.7e-04  & -4.3e-05  \\
4  & 100  &          5254  &           -  &   77.6  &    -  & 1.4e-03  & -  & -6.6e-04  & -6.4e-05  \\
6  & 100  &          5646  &           -  &  188.4  &    -  & 1.1e-03  & -  & -8.6e-04  & -1.0e-04  \\
8  & 100  &          5927  &           -  &  351.3  &    -  & 9.9e-04  & -  & -1.2e-03  & -1.3e-04  \\
10 & 100  &          6631  &           -  &  612.0  &    -  & 9.0e-04  & -  & -1.4e-03  & -1.2e-04  \\
\hline
\multicolumn{1}{c}{$i$} & \multicolumn{9}{l}{\small{$v\IR$$_\mathrm{SPGL1}$}} \\ \cline{2-10}
~ & \small{Success(\%)} & \small{Iter$_\mathrm{s}$} & \small{Iter$_\mathrm{f}$} & \small{CPU$_\mathrm{s}$} & \small{CPU$_\mathrm{f}$} & \small{RecErr$_\mathrm{s}$} &  \small{RecErr$_\mathrm{f}$} & \small{Res$_\mathrm{min}$} & \small{Res$_\mathrm{max}$}\\
\bottomrule
2   & 100  &          1566  &           -  &   10.4  &    -  & 2.0e-03  & -  & -2.7e-04  & 3.5e-05   \\
4   & 100  &          1873  &           -  &   48.2  &    -  & 1.4e-03  & -  & -5.2e-04  & -2.5e-06   \\
6   & 100  &          1880  &           -  &  103.4  &    -  & 1.1e-03  & -  & -7.7e-04  & -6.2e-06   \\
8   & 100  &          1910  &           -  &  186.8  &    -  & 9.9e-04  & -  & -1.2e-03  & -8.9e-05   \\
10  & 100  &          1929  &           -  &  285.7  &    -  & 9.0e-04  & -  & -1.3e-03  & -3.8e-05   \\
\hline
\multicolumn{1}{c}{$i$} & \multicolumn{9}{l}{\small{SCP$_{\rm ls}$}} \\ \cline{2-10}
~ & \small{Success(\%)} & \small{Iter$_\mathrm{s}$} & \small{Iter$_\mathrm{f}$} & \small{CPU$_\mathrm{s}$} & \small{CPU$_\mathrm{f}$} & \small{RecErr$_\mathrm{s}$} &  \small{RecErr$_\mathrm{f}$} & \small{Res$_\mathrm{min}$} & \small{Res$_\mathrm{max}$}\\
 \bottomrule
 2  & 100  &         12757  &           -  &  111.6  &    -  & 2.0e-03  & -  & -1.5e-05  & -9.0e-08  \\
 4  &  50  &         22347  &          8299  &  666.1  &  238.1  & 1.4e-03  & 8.1e-01  & -3.2e-05  & -2.4e-07  \\
 6  &   0  &           -  &          6533  &    -  &  380.6  & -  & 8.4e-01  & -5.3e-05  & -2.1e-05  \\
 8  &   0  &           -  &          5070  &    -  &  507.1  & -  & 8.5e-01  & -6.1e-05  & -2.6e-05  \\
 10 &   0  &           -  &          4529  &    -  &  693.3  & -  & 8.6e-01  & -7.1e-05  & -3.9e-05  \\
\hline
\end{tabular}
\end{center}
\end{table}

\end{document}